\documentclass[12pt,a4paper,reqno]{amsart}
\usepackage[english]{babel}
\usepackage[applemac]{inputenc}
\usepackage[T1]{fontenc}
\usepackage{palatino}
\usepackage{amsmath}
\usepackage{amssymb}
\usepackage{amsthm}
\usepackage{amsfonts}
\usepackage{graphicx}
\usepackage[colorlinks = true, citecolor = black]{hyperref}
\author{Tuomas Orponen}
\thanks{The research was partially supported by the Academy of Finland, grant 133264 "Stochastic and harmonic analysis, interactions and applications", and by the Finnish foundation Jenny ja Antti Wihurin rahasto.}
\title[Restricted families of projections in $\R^{3}$]{Hausdorff dimension estimates for restricted families of projections in $\R^{3}$}
\address{Department of Mathematics and Statistics, University of Helsinki, P.O.B. 68, FI-00014 Helsinki, Finland}
\email{tuomas.orponen@helsinki.fi}

\newcommand{\R}{\mathbb{R}}
\newcommand{\N}{\mathbb{N}}

\newcommand{\Z}{\mathbb{Z}}

\newcommand{\cG}{\mathcal{G}}
\newcommand{\cE}{\mathcal{E}}

\newcommand{\spt}{\operatorname{spt}}

\newcommand{\Pd}{\dim_{\mathrm{p}}}

\newcommand{\cH}{\mathcal{H}}

\newcommand{\cC}{\mathcal{C}}
\newcommand{\spa}{\operatorname{span}}
\newcommand{\diam}{\operatorname{diam}}

\newcommand{\dist}{\operatorname{dist}}

\numberwithin{equation}{section}

\theoremstyle{plain}
\newtheorem{thm}[equation]{Theorem}
\newtheorem{conjecture}[equation]{Conjecture}
\newtheorem{lemma}[equation]{Lemma}

\newtheorem{proposition}[equation]{Proposition}

\newtheorem{fact}[equation]{Fact}

\theoremstyle{definition}
\newtheorem{definition}[equation]{Definition}

\theoremstyle{remark}
\newtheorem{remark}[equation]{Remark}

\begin{document}

\begin{abstract} This paper is concerned with restricted families of projections in $\R^{3}$. Let $K \subset \R^{3}$ be a Borel set with Hausdorff dimension $\dim K = s > 1$. If $\mathcal{G}$ is a smooth and sufficiently well-curved one-dimensional family of two-dimensional subspaces, the main result states that there exists $\sigma(s) > 1$ such that $\dim \pi_{V}(K) \geq \sigma(s)$ for almost all $V \in \mathcal{G}$. A similar result is obtained for some specific families of one-dimensional subspaces. 
\end{abstract}

\maketitle

\tableofcontents

\section{Introduction}

This paper continues a line of research motivated by the question: are there Marstrand-Mattila type projection theorems for restricted families of projections? The original result of J. Marstrand \cite{Mar} and P. Mattila \cite{Mat2} states that if $B \subset \R^{d}$ is an analytic set with Hausdorff dimension $\dim B \leq m$, then $\dim \pi_{V}(B) = \dim B$ for almost all $m$-planes $V \in \mathcal{G}(d,m)$. Here $\pi_{V} \colon \R^{d} \to V$ is the orthogonal projection onto $V$, $\dim$ stands for Hausdorff dimension, and an $m$-plane refers to an $m$-dimensional subspace of $\R^{d}$.

In the 'restricted projections' framework, one chooses a smooth submanifold $\mathcal{G} \subset \mathcal{G}(d,m)$ with $\dim \mathcal{G} < \dim \mathcal{G}(d,m)$ and asks whether $\dim \pi_{V}(B) = \dim B$ for almost all $V \in \mathcal{G}$. To date, several answers are known. First, I mention the results of E. J\"arvenp\"a\"a, M. J\"arvenp\"a\"a, T. Keleti, M. Leikas and F. Ledrappier contained in the papers \cite{JJLL} and \cite{JJK} (the latter of which generalises the theorems in the former). These papers provide a complete answer in the setting where no 'curvature conditions' are placed on $\mathcal{G}$. Indeed, \cite[Theorem 3.2]{JJK} gives an almost sure lower bound for $\dim \pi_{V}(B)$ in terms of $\dim B$ and $\dim \mathcal{G}$. In the typical situation, there exists a number $0 < \sigma < \dim B$, depending on $\dim B$ and $\dim \mathcal{G}$ such that $\dim \pi_{V}(B) \in [\sigma,\dim B]$ for almost every $V \in \mathcal{G}$. Examples in \cite{JJK} show that the lower bounds are sharp.

A natural follow-up question, whether $V \mapsto \dim \pi_{V}(B)$ is almost surely a constant (depending on $B$ and $\mathcal{G}$), was studied by K. F\"assler and the author in \cite{FO1}; positive answers were obtained in some special cases, in particular for the one-dimensional family of planes in $\R^{3}$ containing the $z$-axis. On the other hand, there are some trivial counterexamples, such as the concatenation of the one-dimensional families of planes in $\R^{3}$ containing the $z$-axis and the $x$-axis. 

There is one notable example of a strict submanifold $\mathcal{G} \subset \mathcal{G}(d,m)$, for which it is known that $\dim \pi_{V}(B) = \dim B$ for almost all $V \in \mathcal{G}$, and for all analytic sets $B$ with $\dim B \leq m$. This is the \emph{isotropic Grassmannian} $\mathcal{G} = \mathcal{G}_{h}(d,m)$, a submanifold of $\mathcal{G}(2d,m)$ with positive codimension. The projection theorem for $\mathcal{G}_{h}(d,m)$ is due to Z. Balogh, K. F\"assler, P. Mattila and J. Tyson \cite{BFMT}; a different proof based on the notion of \emph{transversality} was given by R. Hovila \cite{Ho}. 

As mentioned above, the papers \cite{JJLL} and \cite{JJK} do not impose any 'curvature conditions' on the manifold $\mathcal{G}$. In particular, the framework of these papers allows for two counterexamples, which serve well to motivate the definitions below. 
\begin{itemize}
\item[(I)] In the first one, all the $m$-planes in $\mathcal{G}$ are contained in a single non-trivial subspace $W \subset \R^{d}$. Then $\pi_{V}(W^{\perp}) = \{0\}$, for all $V \in \mathcal{G}$, which means that there is no non-trivial dimension conservation result for the projection family $(\pi_{V})_{V \in \mathcal{G}}$. 

\item[(II)] In the second -- and slightly more subtle -- counterexample, the $m$-planes in $\mathcal{G}$ may cover the whole of $\R^{d}$, but they are \emph{co-contained} in a single subspace $W \subset \R^{d}$ with $\dim W \leq m < d$, in the sense that $V^{\perp} \subset W$ for all $V \in \mathcal{G}$. Then $\pi_{V}(W) \subset V \cap W$ for all $V \in \mathcal{G}$. (To see this, pick $w \in W$, and write $w = \pi_{V}(w) + v^{\perp}$ with $v^{\perp} \in V^{\perp} \subset W$. It follows that $\pi_{V}(w) = w - v^{\perp} \in V \cap W$.) In particular, $\dim \pi_{V}(W) < \dim W$ for all $V \in \mathcal{G}$. Since $\dim W \leq m$, this means that $(\pi_{V})_{V \in \mathcal{G}}$ does not satisfy the classical Marstrand-Mattila projection theorem (as formulated in the first paragraph). The simplest case of this type of counter example is the family $\mathcal{G}$ of all planes in $\R^{3}$ containing the $z$-axis.
\end{itemize}

In three dimensions, at least, these are -- essentially -- the only counterexamples known to date. Informally speaking, one could conjecture that any (smooth) one-dimensional family of one- or two-planes, no "large part" of which is contained or co-contained in a single non-trivial subspace, should satisfy the Marstrand-Mattila projection theorem. 

To formulate the hypothesis of the conjecture in precise terms, K. F\"assler and the author proposed in \cite{FO2} the following curvature condition for one-dimensional families of one- and two-planes in $\R^{3}$.
\begin{definition}[Non-degenerate families] Assume that $J \subset \R$ is an open interval, and $\gamma \colon J \to S^{2}$ is a $\cC^{3}$-curve satisfying
\begin{equation}\label{curvature} \spa\{\gamma(\theta),\dot{\gamma}(\theta),\ddot{\gamma}(\theta)\} = \R^{3}, \qquad \theta \in J. \end{equation}
Write $\ell_{\theta} := \spa(\gamma(\theta)) \in \mathcal{G}(3,1)$, and $V_{\theta} := \ell_{\theta}^{\perp} \in \mathcal{G}(3,2)$. Then, the families $\{\ell_{\theta}\}_{\theta \in J}$ and $\{V_{\theta}\}_{\theta \in J}$ are referred to as \emph{non-degenerate families of lines} and \emph{planes}, respectively. The orthogonal projections onto $\ell_{\theta}$ and $V_{\theta}$ are denoted by 
\begin{displaymath} \rho_{\theta} := \pi_{\ell_{\theta}} \quad \text{and} \quad \pi_{\theta} := \pi_{V_{\theta}}, \end{displaymath}
and the families $\{\rho_{\theta}\}_{\theta \in J}$ and $\{\pi_{\theta}\}_{\theta \in J}$ are called \emph{non-degenerate families of projections}.
\end{definition}
In the sequel, $\rho_{\theta}$ and $\pi_{\theta}$ will always refer to members of non-degenerate families of projections. Classical techniques, dating back as far as Kaufman's work \cite{Ka} in 1968, can be used to show that the lower bounds in \cite{JJLL} and \cite{JJK} (obtained without any curvature assumptions) are no longer sharp for the projections $\rho_{\theta}$. In \cite{FO2}, we verified the following proposition:
\begin{proposition}[Proposition 1.4 in \cite{FO2}]\label{prop1} If $B \subset \R^{3}$ is an analytic set, then
\begin{displaymath} \dim \rho_{\theta}(B) \geq \min\left\{\dim B, \frac{1}{2}\right\} \quad \text{and} \quad \dim \pi_{\theta}(B) \geq \min\{\dim B,1\} \quad \text{for a.e. } \theta \in J. \end{displaymath}
\end{proposition}
The lower bound for $\dim \pi_{\theta}(B)$ holds without the curvature condition \eqref{curvature} and was already established in \cite{JJLL}. In contrast, the bounds in \cite{JJLL} and \cite{JJK} give no information about $\dim \rho_{\theta}(B)$ in this situation (at least in case $\dim B \leq 1$) -- the reason being example (I) above. But even if Proposition \ref{prop1} improves on \cite{JJLL} and \cite{JJK} under the curvature hypothesis \eqref{curvature}, there is no longer reason to believe that the bounds $\min\{\dim B,1/2\}$ and $\min\{\dim B,1\}$ are sharp. The guess that they are not is the content of the following conjecture, formalising the discussion above:
\begin{conjecture} If $B \subset \R^{3}$ is an analytic set, then
\begin{displaymath} \dim \rho_{\theta}(B) = \min\{\dim B, 1\} \quad \text{and} \quad \dim \pi_{\theta}(B) = \min\{\dim B,2\} \quad \text{for a.e. } \theta \in J. \end{displaymath}
\end{conjecture} 
The main results in \cite{FO2} were the verification of the first part of this conjecture for self-similar sets in $\R^{3}$ without rotations, and a slight improvement over the $\min\{\dim B,1/2\}$ and $\min\{\dim B,1\}$ bounds for \emph{packing dimension} $\Pd$:
\begin{thm}[Theorem 1.6 in \cite{FO2}]\label{packingDimension} Assume that $B \subset \R^{3}$ is an analytic set with $\dim B = s$. If $s > 1/2$, there exists a constant $\sigma_{1}(s) > 1/2$ such that
\begin{displaymath} \Pd \rho_{\theta}(B) \geq \sigma_{1}(s) \quad \text{for a.e. } \theta \in J. \end{displaymath} 
If $s > 1$, there exists a constant $\sigma_{2}(s) > 1$ such that
\begin{displaymath} \Pd \pi_{\theta}(B) \geq \sigma_{2}(s) \quad \text{for a.e. } \theta \in J. \end{displaymath}
\end{thm}
The appearance of $\Pd$ in the theorem above was unfortunate, but the method of proof simply did not yield the same conclusion for $\dim$. The first version of the present paper addressed the issue in the special case, where $\rho_{\theta}$ and $\pi_{\theta}$ are obtained from the curve
\begin{equation}\label{specialCurve} \gamma(\theta) = \tfrac{1}{\sqrt{2}}(\cos \theta, \sin \theta, 1), \qquad \theta \in (0,2\pi). \end{equation}
In other words, Theorem \ref{packingDimension} was proven with $\Pd$ replaced by $\dim$, but only for these specific families of projections. The first results for Hausdorff dimension in the general situation were, soon afterwards, obtained by D. and R. Oberlin \cite{OO}. Here is their main result:
\begin{thm}[Theorem 1.1 in \cite{OO}]\label{Oberlin} Let $B \subset \R^{3}$ be an analytic set with $\dim B \geq 1$. Then, for almost every $\theta \in J$, one has the lower bounds
\begin{displaymath} \dim \pi_{\theta}(B) \geq \begin{cases} (3/4) \dim B, & \text{if } 1 \leq \dim B \leq 2,\\ \min\{\dim B - 1/2, 2\}, & \text{if } 2 \leq \dim B \leq 3. \end{cases} \end{displaymath}

\end{thm}
The proof of Theorem \ref{Oberlin} is based on a Fourier restriction estimate. One should note that the technique does not seem to yield improvements over the bound $\min\{\dim B,1\}$ bound, when $1 \leq \dim B \leq 4/3$. Such an improvement is the main result of this paper:
\begin{thm}\label{main} Let $B \subset \R^{3}$ be an analytic set with $\dim B = s > 1$. There exists a constant $\sigma(s) > 1$ such that
\begin{displaymath} \dim \pi_{\theta}(B) \geq \sigma(s) \quad \text{for a.e. } \theta \in J. \end{displaymath} 
\end{thm}

For the projections onto lines, the result is analogous but only concerns the specific family arising from the curve \eqref{specialCurve}:
\begin{thm}\label{main2} Assume that $\rho_{\theta}$ is the orthogonal projection onto the line $\ell_{\theta} = \spa(\gamma(\theta))$, where $\gamma$ is the curve from \eqref{specialCurve}. Let $B \subset \R^{3}$ be an analytic set with $\dim B = s > 1/2$. There exists a constant $\tilde{\sigma}(s) > 1/2$ such that
\begin{displaymath} \dim \rho_{\theta}(B) \geq \tilde{\sigma}(s) \quad \text{for a.e. } \theta \in (0,2\pi). \end{displaymath}
\end{thm} 

The manner in which $\sigma(s)$ and $\tilde{\sigma}(s)$ are derived would, in principle, allow for their explicit determination, but I will not pursue this below. Speaking off the record, it seems likely that one could obtain $\sigma(s) = 1 + O((s - 1)^{2})$ for $s$ near $1$, and $\tilde{\sigma}(s) = 1/2 + O((s - 1/2)^{2})$ for $s$ close to $1/2$.

It is worth explaining, why our results do not follow from the work of Y. Peres and W. Schlag \cite{PS} on \emph{generalized projections}. The reason is that a one-dimensional family of orthogonal projections onto lines or planes in $\R^{3}$ can never satisfy the key \emph{transversality} property of generalized projections. For projections onto planes, this is so simply because the notion of transversality (see \cite[Definition 7.1]{PS}) requires the dimension of the target space of the projections, here $2$, to be at most as large as the dimension of the parameter set, here $1$. For projections onto lines, this necessary condition is naturally satisfied: however, if $\{\rho_{\theta}\}_{\theta \in J}$ was a one-dimensional family of orthogonal projections onto lines in $\R^{3}$ satisfying the generalized projections framework (this time, see \cite[Definition 2.7]{PS}), then, as shown in \cite[Remark 2.9]{O}, one could easily devise an "almost bi-Lipschitz" mapping between $\R^{3}$ and $\R^{2}$, which is absurd.

The paper is structured as follows. Section $3$ is concerned with projections onto planes -- namely the proof of Theorem \ref{main} -- whereas Section $4$ treats projections onto lines and contains the proof of Theorem \ref{main2}. The latter half of the paper consists of two appendices, which contain the proofs of certain geometric lemmas required in the earlier sections.

I close the introduction with a word on notation. For technical purposes, it is convenient to view $\rho_{\theta}$ and $\pi_{\theta}$ as mappings from $\R^{3}$ to $\R$ and $\R^{2}$, respectively. Throughout the paper I will write $a \lesssim b$, if $a \leq Cb$ for some constant $C \geq 1$. The two-sided inequality $a \lesssim b \lesssim a$, meaning $a \leq C_{1}b \leq C_{2}a$, is abbreviated to $a \sim b$. Should I wish to emphasise that the implicit constants depend on a parameter $p$, I will write $a \lesssim_{p} b$ and $a \sim_{p} b$. The closed ball in $\R^{d}$ with centre $x$ and radius $r > 0$ will be denoted by $B(x,r)$. For $A \subset \R^{d}$ and $\delta > 0$, I denote by $B(A,\delta) := \{x \in \R^{d} : \dist(x,A) \leq \delta\}$ the closed $\delta$-neighbourhood of $A$. The notation $|E|$ refers to the Lebesgue measure of a set $E \subset \R^{d}$.

\section{Acknowledgements}

I am grateful to Katrin F\"assler for many discussions and comments on the manuscript, and for supplying the proof contained in Remark \ref{parametrisationRemark}. This paper would definitely not exist without our earlier work on the subject. I also wish to thank the anonymous referee for reading the manuscript very carefully and giving many comments, which helped me clarify -- in some cases even simplify -- parts of the text. 

\section{Projections onto planes}

The proofs of Theorems \ref{main} and \ref{main2} have a lot in common, but the former is technically simpler. So, I start there. 

\begin{proof}[Proof of Theorem \ref{main}] Let $B \subset \R^{3}$ be an analytic set with $\dim B = s > 1$. Make a counter assumption: there exists a compact set $E \subset J$ with positive length such that
\begin{equation}\label{counter1} \cH^{\sigma}(\pi_{\theta}(B)) \leq 1 \quad \text{for all } \theta \in E. \end{equation}
The parameter $\sigma \in (1,s)$ will be fixed during the proof; in the end, it will only depend on $s$ how close $\sigma$ has to be chosen to one. Roughly speaking, the plan is to extract structural information about $B$ based on our counter assumption -- and to show that if $\sigma$ is close to one, no $s$-dimensional set can have such structure. 

There will be occasions, when it is required or useful to assume that $J$ is "short enough" for various purposes. Since this can always be done without loss of generality -- by covering $J$ by short subintervals and proving the theorem individually for those -- I will not make further remark about the issue.

The first task is to find small 'bad' scales $\delta > 0$, where the counter assumption \eqref{counter1} has a tractable geometric interpretation. This pigeonholing argument is essentially the same as \cite[p. 222]{Bo} by Bourgain. 

\begin{lemma}\label{cover} Let $A \subset \R^{d}$ be a set with $\cH^{\sigma}(A) \leq 1$. Then, for any $\delta_{0} > 0$, there exist collections of balls $\cG_{k}$, $2^{-k} < \delta_{0}$, with the properties that \emph{(i)} the balls in $\cG_{k}$ have bounded overlap (that is, the sum of their indicators is a bounded function), \emph{(ii)} they have diameter $\sim 2^{-k}$, \emph{(iii)} there are no more than $\lesssim_{d} 2^{k\sigma}$ balls in $\cG_{k}$, and
\begin{equation} A \subset \bigcup_{2^{-k} < \delta_{0}} \bigcup_{B' \in \cG_{k}} B'. \tag{iv} \end{equation}
\end{lemma}

\begin{proof} By the very definition of $\cH^{\sigma}(A) \leq 1$, one may find collections of balls $\cG^{0}_{k}$, $2^{-k} < \delta_{0}$, satisfying conditions (ii)--(iv). In order to have (i), one first uses the $5r$-covering theorem to extract a disjoint subcollection $\cG_{k}^{1} \subset \cG^{0}_{k}$ such that 
\begin{displaymath} \bigcup_{B' \in \cG^{0}_{k}} B' \subset \bigcup_{B' \in \cG_{k}^{1}} 5B'. \end{displaymath}
Now, the collection $\cG_{k} := \{5B' : B' \in \cG_{k}^{1}\}$ satisfies all the requirements. \end{proof}


Fix $\delta_{0} > 0$ and $\theta \in E$. Based on the counter assumption \eqref{counter1} and the lemma above, find collections $\cG_{\theta,k}$, $2^{-k} < \delta_{0}$, of discs in $\R^{2}$ such that the properties (i)--(iv) listed in the lemma are satisfied with $A = \pi_{\theta}(B) \subset \R^{2}$. Without loss of generality (by Frostman's lemma), assume that $B = \spt \mu \subset B(0,1)$, where $\mu$ is a Borel probability measure on $\R^{3}$ satisfying
\begin{displaymath} \mu(B(x,r)) \lesssim r^{s} \quad \text{for } x \in \R^{3} \text{ and } r > 0. \end{displaymath}
Then, Lemma \ref{cover} (iv) implies that
\begin{displaymath} \sum_{2^{-k} < \delta_{0}} \mu(\pi_{\theta}^{-1}(\cup \, \cG_{k,\theta})) \geq 1, \end{displaymath}
where $\cup \, \cG_{k,\theta}$ stands for the union of the discs in $\cG_{k,\theta}$. In particular, there exists $k \in \N$ with $2^{-k} < \delta_{0}$ such that
\begin{equation}\label{form2} \mu\left(\pi_{\theta}^{-1}\left(\cup \, \cG_{k,\theta} \right)\right) \gtrsim k^{-2}. \end{equation}
Since the conclusion holds for every $\theta \in E$, one may further pigeonhole $k \in \N$ so that \eqref{form2} holds for all $\theta \in E_{k} \subset E$, where $|E_{k}| \gtrsim_{|E|} k^{-2}$. For this $k \in \N$, write $\delta := 2^{-k} < \delta_{0}$, $\cG_{\theta} := \cG_{k,\theta}$ and $E_{\delta} := E_{k}$. In the sequel, whenever the text says 'by taking $\delta > 0$ small enough' or something similar, one should understand it as 'first choose $\delta_{0} > 0$ small enough, and then run through the pigeonholing argument above to find $\delta < \delta_{0}$'.

Given $\theta \in [0,2\pi)$ and $x,y \in \R^{3}$, define the relation $x \sim_{\theta} y$ by
\begin{displaymath} x \sim_{\theta} y \quad \Longleftrightarrow \quad x,y \in \pi_{\theta}^{-1}(B) \text{ for some } B \in \cG_{\theta}. \end{displaymath}
So, the condition $x \sim_{\theta} y$ means that $x$ and $y$ share a common '$\delta$-tube' in $\R^{3}$. We now define the energy $\cE$ by
\begin{displaymath} \cE := \int_{0}^{2\pi} \mu \times \mu(\{(x,y) : x \sim_{\theta} y\}) \, d\theta = \iint |\{\theta \in [0,2\pi) : x \sim_{\theta} y\}| \, d\mu x \, d\mu y. \end{displaymath}
The next aim is to bound $\cE$ from below; this will be accomplished using the first expression above. Fix $\theta \in E_{\delta}$. Then \eqref{form2} holds, so there is a collection of $\delta$-tubes $T_{1},\ldots,T_{N}$ of the form $T_{j} = \pi_{\theta}^{-1}(B_{j})$, $B_{j} \in \cG_{\theta}$, such that the total $\mu$-mass of the tubes $T_{j}$ is $\gtrsim (\log 1/\delta)^{-2}$, and $N \lesssim \delta^{-\sigma}$. For each $T_{j}$, one has $T_{j} \times T_{j} \subset \{(x,y) : x \sim_{\theta} y\}$. Using this fact, the bounded overlap of the product sets $T_{j} \times T_{j}$ and the Cauchy-Schwarz inequality, one obtains the following estimate:
\begin{align*} \mu \times \mu(\{(x,y) : x \sim_{\theta} y\}) & \gtrsim \sum_{j = 1}^{N} \, [\mu(T_{j})]^{2}\\
& \geq \frac{1}{N} \left(\sum_{j = 1}^{N} \mu(T_{j}) \right)^{2}\\
& \gtrsim \delta^{\sigma} \cdot \mu\left(\bigcup_{j = 1}^{N} T_{j} \right)^{2}\\
& \gtrsim \delta^{\sigma} \cdot \left(\log \left(\frac{1}{\delta} \right) \right)^{-4}. \end{align*}
Integrating over $\theta \in E_{\delta}$ and recalling that $|E_{\delta}| \gtrsim (\log (1/\delta))^{-2}$ yields
\begin{equation}\label{form3} \cE \gtrsim \delta^{\sigma} \cdot \left(\log \left(\frac{1}{\delta} \right)\right)^{-6}. \end{equation}

The next question is: what structural information about $B = \spt \mu$ does \eqref{form3} provide? Write
\begin{displaymath} \cC := \bigcup_{\theta \in J} B(\ell_{\theta},\delta), \end{displaymath}
where $\ell_{\theta} := \spa(\gamma(\theta)) = \pi_{\theta}^{-1}\{0\}$. Thus, $\cC$ is the closed $\delta$-neighbourhood of a "conical" surface in $C \subset \R^{3}$. This intuition is correct, if the $\gamma(J)$ is contained in a small disc in a single hemisphere of $S^{2}$. This can be assumed without loss of generality, and such an assumption is indeed required a little later.  The rest of the proof runs as follows. If $\delta>0$ is small, one uses \eqref{form3} to find two points $x_{1},x_{2} \in \R^{3}$ such that
\begin{equation}\label{form004} |x_{1} - x_{2}| \geq \delta^{\kappa}, \end{equation}
and
\begin{equation}\label{form4} \mu((x_{1} + \cC) \cap (x_{2} + \cC)) \geq \delta^{\kappa}. \end{equation}
Here $\kappa > 0$ is a number depending on $s$ and $\sigma$ with the crucial property that it can be chosen arbitrarily close to zero by letting $\sigma \searrow 1$. On the other hand, there is Lemma \ref{twoConesLemma} below, stating (informally speaking) that if two conical surfaces -- such as $C$ -- in $\R^{3}$ are well separated, then the intersection of their $\delta$-neighbourhoods behaves like a one-dimensional object. But $\mu$ is a Frostman measure with index $s > 1$, so such objects cannot have so much mass as \eqref{form4} postulates for small $\kappa$. This will, eventually, show that \eqref{form004} and \eqref{form4} are mutually incompatible and conclude the proof.


The hunt for the points $x_{1},x_{2} \in \R^{3}$ begins. First, observe that
\begin{equation}\label{form5} \cE = \int \int_{y + \cC} |\{\theta \in [0,2\pi) : x \sim_{\theta} y\}| \, d\mu x \, d\mu y. \end{equation}
Indeed, if $x \notin y + \cC$, then the distance of $x$ to any of the lines $y+\ell_{\theta}$, $\theta \in [0,2\pi)$, is greater than $\delta$, and consequently $|\pi_{\theta}(x - y)| > \delta$ for all $\theta \in [0,2\pi)$. In particular, $x \not\sim_{\theta} y$ for all $\theta \in [0,2\pi)$. To estimate the integral in \eqref{form5} further, the following universal bound is needed:
\begin{lemma}\label{lemma1} If $x,y \in \R^{3}$ are distinct points, then
\begin{displaymath} |\{\theta \in [0,2\pi) : x \sim_{\theta} y\}| \lesssim \frac{\delta}{|x - y|}. \end{displaymath}
\end{lemma}
\begin{proof} Observe that
\begin{displaymath} \{\theta \in [0,2\pi) : x \sim_{\theta} y\} \subset \{\theta \in [0,2\pi) : |\pi_{\theta}(x - y)| \leq \delta\}. \end{displaymath}
The length of the set on the right hand side can be estimated by studying the smooth function $\theta \mapsto |\pi_{\theta}(\xi)|^{2}$, $\xi \in S^{2}$. The crucial observation is that this function can have at most second order zeros. The details can be found above \cite[(3.9)]{FO2}.  \end{proof}
Now, in order to estimate the right hand side of \eqref{form5},  define
\begin{displaymath} G := \{y \in \R^{3} : \mu(y + \cC) \geq \delta^{\tau}\}, \end{displaymath}
where $\tau = \kappa/5 > 0$. Write
\begin{align*} \cE = & \int_{G} \int_{y + \cC} |\{ \theta \in [0,2\pi) : x \sim_{\theta} y\}| \, d\mu x \, d\mu y\\
& \quad + \int_{\mathbb{R}^3\setminus G} \int_{y + \cC} |\{ \theta \in [0,2\pi) : x \sim_{\theta} y\}| \, d\mu x \, d\mu y. \end{align*}
The terms will be referred to as $I_{G}$ and $I_{\mathbb{R}^3\setminus G}$. The term $I_{G}$ is estimated using the bound from Lemma \ref{lemma1}, and recalling the uniform bound $\mu(B(x,r)) \lesssim r^{s}$, $s > 1$:
\begin{equation}\label{form6} I_{G} \lesssim \delta \cdot \int_{G} \int \frac{1}{|x - y|} \, d\mu x \, d\mu y \lesssim_{s} \delta \cdot \mu(G).  \end{equation}
In order to estimate $I_{\mathbb{R}^3\setminus G}$, write $A_{j}(y) := \{x \in \R^{3} : 2^{j} \leq |x - y| \leq 2^{j + 1}\}$. For every $j \in \Z$ with $\delta \leq 2^{j} \leq 1$, couple the bound from Lemma \ref{lemma1} with the estimate $\mu((y + \cC) \cap A_{j}(y)) \lesssim \min\{\delta^{\tau},2^{js}\} \leq \delta^{\tau(1 - 1/s)} \cdot 2^{j}$, valid for $y \in \R^{3} \setminus G$.
\begin{align*} I_{\mathbb{R}^3\setminus G} & \lesssim \int_{\mathbb{R}^3\setminus G} \int_{B(y,\delta)} \, d\mu x \, d\mu y\\
& \quad + \int_{\mathbb{R}^3\setminus G} \sum_{\delta \leq 2^{j} \leq 1} \int_{(y + \cC) \cap A_{j}(y)} |\{ \theta \in [0,2\pi) : x \sim_{\theta} y\}| \, d\mu x \, d\mu y\\
& \lesssim \delta^{s} + \delta \cdot \int_{\mathbb{R}^3\setminus G} \sum_{\delta \leq 2^{j} \leq 1} 2^{-j} \cdot \mu((y + \cC) \cap A_{j}(y)) \, d\mu y\\
& \lesssim \delta^{s} + \delta^{1 + \tau(1 - 1/s)} \cdot \log\left(\frac{1}{\delta}\right). \end{align*}
Comparing the upper bounds for $I_{G}$ and $I_{\mathbb{R}^3\setminus G}$ with the lower bound \eqref{form3} results in
\begin{displaymath} \delta^{\sigma} \cdot \left(\log \left(\frac{1}{\delta} \right)\right)^{-6} \lesssim \delta \cdot \mu(G) + \delta^{s} + \delta^{1 + \tau(1 - 1/s)} \cdot \log\left(\frac{1}{\delta}\right). \end{displaymath}
One of the three terms on the right hand side must dominate the left hand side. The middle term clearly can never do that, since $\sigma <s$. Neither can the last term, if one chooses $\sigma < 1 + \tau(1-1/s) < 1 + \tau$. Then, the only possibility remaining is that
\begin{displaymath} \mu(G) \gtrsim \delta^{\sigma - 1} \cdot \left(\log \left(\frac{1}{\delta} \right)\right)^{-6} \gtrsim \delta^{\tau}. \end{displaymath}
In other words, if the counter assumption is strong enough ($\sigma$ is close enough to one), the 'good set' $G$ has relatively large $\mu$ measure. Now, a small technical point: the conical surface $C$ was "two-sided" to begin with, but later it will be easier to deal with just a one-sided versions of $C$ and $\cC$. So, define $\cC^{+} := \cC \cap \{(x,y,h) : h \geq 0\}$ and and note that $\cC \setminus \cC^{+} \subset -\cC^{+}$. Then, the estimate $\mu(G) \geq \delta^{\tau}$ implies that either $\mu(G^{+}) \geq \delta^{\tau}/2$ or $\mu(G^{-}) \geq \delta^{\tau}/2$, where
\begin{displaymath} G^{+} := \{y \in \R^{3} : \mu(y + \cC^{+}) \geq \delta^{\tau}/2\} \quad \text{and} \quad G^{-} := \mu(\{y \in \R^{3} : \mu(y - \cC^{+}) \geq \delta^{\tau}/2\}. \end{displaymath}
Assume, for instance, that $\mu(G^{-}) \geq \delta^{\tau}/2$. This will easily yield the existence of the points $x_{1},x_{2} \in \R^{3}$. First, one uses H\"older's inequality to make the following estimate:
\begin{align*} A & := \iint \mu((x_{1} + \cC^{+}) \cap (x_{2} + \cC^{+})) \, d\mu x_{1} \, d\mu x_{2}\\
& = \iint \int \chi_{x_{1} + \cC^{+}}(y)\chi_{x_{2} + \cC^{+}}(y) \, d\mu y \, d\mu x_{1} \, d\mu x_{2}\\
& = \int \mu(y - \cC^{+})^{2} \, d\mu y \geq \left(\int \mu(y - \cC^{+}) \, d\mu y \right)^{2} \gtrsim \delta^{4\tau}.  \end{align*}
Recall that the aim is to find two points $x_{1},x_{2} \in \spt \mu \subset B(0,1)$ such \eqref{form4} holds -- with $\cC$ replaced by $\cC^{+}$, in fact -- and the mutual distance of the points $x_{i}$ is at least $\delta^{\kappa} = \delta^{5\tau}$. If this cannot be done, then
\begin{displaymath} |x_{1} - x_{2}| \geq \delta^{5\tau} = \delta^{\kappa} \; \Longrightarrow \; \mu((x_{1} + \cC^{+}) \cap (x_{2} + \cC^{+})) < \delta^{5\tau} = \delta^{\kappa} \end{displaymath}
for all $x_{1},x_{2} \in \spt \mu$. Thus,
\begin{displaymath} A \leq \int \int_{B(x_{2},\delta^{5\tau})} 1 \, d\mu x_{1} \, d\mu x_{2} + \iint_{\{|x_{1} - x_{2}| \geq \delta^{5\tau}\}} \delta^{5\tau} \, d\mu x_{1} \, d\mu x_{2} \lesssim \delta^{5s\tau} + \delta^{5\tau}. \end{displaymath}
Since $s > 1$, for small enough $\delta > 0$ this violates the lower bound for $A$ obtained above. The conclusion is that there exist points $x_{1},x_{2},x_{3} \in B(0,1)$ satisfying \eqref{form004} and \eqref{form4} with $\cC$ replaced by $\cC^{+}$. For simplicity of notation, assume that $x_{1} = 0$.

Now, it is time to state the main geometric lemma. The proof is a bit technical, so it is postponed to Appendix \ref{AppendixB}.

\begin{lemma}[Two cones lemma]\label{twoConesLemma} The following holds for small enough $\epsilon > 0$, for all short enough intervals $J \subset \R$ (the precise requirements will be explained in the appendix), for small enough $\delta > 0$, and for $5\epsilon \leq \tau < 1/2$. If $p \in \R^{3}$ is a point with $|p| \geq \delta^{\epsilon}$, then the intersection
\begin{displaymath} \cC^{+} \cap (\cC^{+} + p) \cap B(0,1) \end{displaymath}
can be covered by two balls of diameter $\lesssim \delta^{\epsilon}$, plus either 
\begin{itemize}
\item[(a)] $\lesssim \delta^{-1/2 - 2\tau - R\epsilon}$ balls of diameter $\lesssim \delta^{1/2 - R\epsilon}$, or
\item[(b)]  $\lesssim \delta^{-\tau/4 - R\epsilon}$ balls of diameter $\lesssim \delta^{\tau/4 - R\epsilon}$, 
\end{itemize}
where $R \geq 1$ is an absolute constant.
\end{lemma}


The correct interpretation is that either option (a) or (b) holds depending on $p$ -- and not that one can choose at will between them. Assuming the lemma, the proof of Theorem \ref{main} is completed as follows. Apply the lemma with $\epsilon = \kappa$ and $p = x_{2}$. Since $\mu$ is a measure on $\R^{3}$ with $\mu(B) \leq \diam(B)^{s}$ for all balls $B$ and for some $s > 1$, the lemma shows that
\begin{align} \mu(\cC^{+} \cap (\cC^{+} + x_{2})) & \lesssim \notag \delta^{\kappa s} + \delta^{-1/2 - 2\tau - R\kappa} \cdot \delta^{s(1/2 - R\kappa)} + \delta^{-\tau/4 - R\kappa} \cdot \delta^{s(\tau/4 - R\kappa)}\\
&\label{muEstimate} = \delta^{\kappa s} + \delta^{(s - 1)/2 - 2\tau - R\kappa(s + 1)} + \delta^{\tau(s - 1)/4 - R\kappa(s + 1)}.  \end{align}
for small enough $\delta > 0$, and for any $1/2 > \tau \geq 5\kappa$. It remains to fix the parameters. First choose $\kappa_{0} > 0$ so small that $(s - 1)/2 - 3 \cdot 5\kappa_{0} > 2\kappa_{0}$. Next, pick $\kappa \in (0,\kappa_{0})$ so small that 
\begin{displaymath} (s - 1)/2 - 2 \cdot 5\kappa_{0} - R\kappa(s + 1) \geq (s - 1)/2 - 3 \cdot 5\kappa_{0} > 2\kappa_{0} \end{displaymath}
and
\begin{displaymath} 5\kappa_{0}(s - 1)/4 - R\kappa(s + 1) \geq 5\kappa_{0}(s - 1)/8 \geq 2\kappa. \end{displaymath}
The left hand sides of these inequalities correspond to exponents in \eqref{muEstimate} with $\tau = 5\kappa_{0}$. Finally, since Lemma  \ref{twoConesLemma} states that \eqref{muEstimate} holds for \textbf{any} $\tau \in [5\kappa,1/2)$ under the assumption $|x_{2}| \geq \delta^{\kappa}$, the inequality holds for $\tau := 5\kappa_{0} \geq 5\kappa$ in particular. This leads to $\mu(\cC^{+} \cap (\cC^{+} + x_{2})) \lesssim \delta^{\eta}$ whenever $|x_{2}| \geq \delta^{\kappa}$ where
\begin{displaymath} \eta = \min\{\kappa s,2\kappa_{0},2\kappa\} > \kappa, \end{displaymath} 
contradicting the choice of $x_{2}$ and concluding the proof. \end{proof}

\section{Projections onto lines}

Theorem \ref{main2} is established in this section. As a quick reminder, it concerns the projections $\rho_{\theta} \colon \R^{3} \to \R$ onto the lines $\ell(\theta) := \spa(\gamma(\theta))$, where
\begin{displaymath} \gamma(\theta) = \frac{1}{\sqrt{2}}(\cos \theta, \sin \theta, 1). \end{displaymath} 
The lines $\ell_{\theta}$ foliate the (classical) conical surface $C = \{(x,y,z) : x^{2} + y^{2} = z^{2}\}$. 

On first sight, it appears that an argument of the kind used in the previous section cannot work for the projections $\rho_{\theta} \colon \R^{3} \to \R$. In fact, one can still make a counter assumption -- this time that $\dim \rho_{\theta}(B) \approx 1/2$ for many $\theta \in [0,2\pi)$ -- and use it to find two well-separated copies of $C$ with the property that a large portion of $B$ is contained in the $\delta$-neighbourhoods of \textbf{both} surfaces. This time, however, there is no contradiction: the sets $B$ one is (mainly) interested in have $\dim B \leq 1$, so they can be easily contained in $(C + p) \cap (C + q)$ for some $p \neq q$. 

The first new idea is to use three copies of $C$ instead of two. The difference in the argument is mostly cosmetic, and the upshot is that one finds three well-separated points $p,q,r \in \R^{3}$ such that a large part of $B$ is contained near $p + C, q + C$ and $r + C$ each. Now, the intersection $(p + C) \cap (q + C) \cap (r + C)$ is either empty or contained in a line, which -- after a lengthy geometric argument given in Appendix \ref{AppendixA} -- shows that the intersection of the $\delta$-neighbourhoods of $p + C, q + C$ and $r + C$ is contained in the neighbourhood of a line on $p + C$, with quantitative bounds. Still, there is no contradiction, since $B$ could actually be contained in such a line. It is also worth noting that increasing the number of intersections beyond three gives no new information.

The final new trick is to start the whole proof by asking: if $B$ is fixed, how many parameters $\theta \in [0,2\pi)$ can there be such that a large part of $B$ is contained near $p + \ell_{\theta}$ for some $p$? It feels intuitive that there cannot be many such values of $\theta$, and this is not hard to prove either: roughly speaking, the "bad" parameters $\theta$ have measure zero. After this is established, the counter assumption $\dim \rho_{\theta}(B) \approx 1/2$ must also hold for positively many "good" $\theta$. Finally, one can replace $C$ by
\begin{displaymath} C_{\text{good}} := \bigcup_{\theta \text{ is good}} \ell_{\theta} \end{displaymath}
and run the argument through -- with three copies of $C_{\text{good}}$ -- as outlined above. The conclusion is that a large part of $B$ must be contained near $p + C_{\text{good}}, q + C_{\text{good}}$ and $r + C_{\text{good}}$ each, and then it follows that a large part of $B$ is contained near $p + \ell_{\theta}$ for some "good" $\theta$. This is a contradiction.

\begin{proof}[Proof of Theorem \ref{main2}] For the  "final new trick" described above, one needs to consider the family of projections onto planes $\tilde{\pi}_{\theta} \colon \R^{3} \to \tilde{V}_{\theta}$, $\theta \in [0,2\pi)$, where
\begin{displaymath} \tilde{V}_{\theta} = \spa(b_{\theta})^{\perp}, \end{displaymath}
and $b_{\theta}$ is the line $b_{\theta} = \spa(\gamma(\theta) \times \dot{\gamma}(\theta)) = \spa((\cos \theta, \sin \theta, -1))$. As before, it suffices to prove Theorem \ref{main2} in the case $B = \spt \mu \subset B(0,1)$, where $\mu$ is a Borel probability measure on $\R^{3}$ satisfying
\begin{displaymath} I_{s}(\mu) := \iint \frac{d\mu x \, d\mu y}{|x - y|^{s}} < \infty. \end{displaymath}
and the growth condition $\mu(B(x,r)) \lesssim r^{s}$ for all balls $B(x,r) \subset \R^{3}$. Moreover, one may assume that $1/2 < s < 1$. Under these hypotheses,
\begin{equation}\label{form9} \int_{0}^{2\pi} I_{s}(\tilde{\pi}_{\theta\sharp}\mu) \, d\theta < \infty, \end{equation}
where $\tilde{\pi}_{\theta\sharp}\mu$ is the measure on $\tilde{V}_{\theta}$ defined by $\tilde{\pi}_{\theta\sharp}\mu(A) = \mu(\tilde{\pi}_{\theta}^{-1}(A))$. Indeed, the finiteness of the integral in \eqref{form9} follows from the sub-level set estimate
\begin{displaymath} |\{\theta : |\tilde{\pi}_{\theta}(x)| \leq \lambda\}| \lesssim \lambda, \end{displaymath} 
valid for all $x \in S^{2}$ and all sufficiently small $\lambda > 0$. For more details on how to prove \eqref{form9}, see \cite[\S3.1]{FO2}, in particular \cite[(3.9)]{FO2}. 

From \eqref{form9}, one sees that $|\{\theta : I_{s}(\tilde{\pi}_{\theta\sharp}\mu) \geq C\}| \to 0$ as $C \to \infty$. Combining this fact with a counter assumption to Theorem \ref{main2}, one finds a constant $C > 0$ and a compact positive length set $E \subset [0,2\pi)$ with the properties that 
 \begin{equation}\label{form10} I_{s}(\tilde{\pi}_{\theta\sharp}\mu) \leq C, \qquad \theta \in E, \end{equation}
and
\begin{equation}\label{form012} \cH^{\sigma}(\rho_{\theta}(B)) \leq 1, \qquad \theta \in E. \end{equation}
This time, $\sigma > 1/2$ is a parameter close to $1/2$, to be fixed in the course of the proof. The assumption \eqref{form10} is the "final new trick": it guarantees that tubes perpendicular to the planes $\tilde{V}_{\theta}$ cannot carry too much $\mu$ mass. This is quantified by the following lemma:
 \begin{lemma}\label{littleMass} Let $\nu$ be a probability measure on $\R^{2}$. Then 
 \begin{displaymath} \nu(B) \leq I_{s}(\nu)^{1/2}\diam(B)^{s/2} \end{displaymath}
 for all $\nu$-measurable sets $B \subset \R^{2}$.
 \end{lemma} 
 
 \begin{proof} Observe that
 \begin{equation}\label{form11} \int_{0}^{\infty} \nu \times \nu(\{(x,y) : |x - y|^{-s} \geq \lambda\}) \, d\lambda = I_{s}(\nu). \end{equation} 
 Now, let $B \subset \R^{2}$ be a $\nu$-measurable set. Then, as long as $x,y \in B$ and $\lambda \leq \diam(B)^{-s}$, one has $|x - y|^{-s} \geq \diam(B)^{-s} \geq \lambda$. This yields the lower bound
 \begin{displaymath} \int_{0}^{\infty} \nu \times \nu(\{(x,y) : |x - y|^{-s} \geq \lambda\}) \, d\lambda \geq \int_{0}^{\diam(B)^{-s}} [\nu(B)]^{2} \, d\lambda = \diam(B)^{-s} \cdot [\nu(B)]^{2}. \end{displaymath}
 A comparison with \eqref{form11} completes the proof.
 \end{proof}
 
It follows from \eqref{form10} and the lemma, that if $T$ is an $\epsilon$-tube perpendicular to a plane $\tilde{V}_{\theta}$, $\theta \in E$, then $\mu(T) \lesssim \epsilon^{s/2}$. However, given the counter assumption \eqref{form012}, and assuming that $\sigma$ is very close to $1/2$, one can extract such tubes $T$ with mass far greater than $\epsilon^{s/2}$. This contradiction will complete the proof in the end.

The search for these 'bad' tubes $T$ begins much like the search for the translated cones $x + \cC$, as seen in the proof of Theorem \ref{main}. The first step is to fix $\delta_{0} > 0$ and find a 'bad' scale $\delta < \delta_{0}$ as before. This process is repeated practically verbatim, so I only state the conclusion. There exists a scale $\delta < \delta_{0}$, a set $E_{\delta} \subset E$, and collections of intervals $\cG_{\theta}$, $\theta \in E_{\delta}$, such that
\begin{itemize}
\item[(i)] for every $\theta \in E_{\delta}$, the collection $\cG_{\theta}$ consists of $\lesssim \delta^{-\sigma}$ intervals with length $\sim \delta$ and bounded overlap,
\item[(ii)] $E_{\delta}$ is compact, and
\begin{displaymath} |E_{\delta}| \gtrsim \left(\log \left(\frac{1}{\delta}\right) \right)^{-2}, \end{displaymath}
\item[(iii)]
\begin{displaymath} \mu\left(\rho_{\theta}^{-1}(\cup \cG_{\theta}) \right) \gtrsim \left(\log \left(\frac{1}{\delta}\right) \right)^{-2} \quad \text{for } \theta \in E_{\delta}. \end{displaymath}
\end{itemize}

The relation $x \sim_{\theta} y$, for $x,y \in \R^{3}$, is defined analogously with the earlier notion:
\begin{displaymath} x \sim_{\theta} y \quad \Longleftrightarrow \quad x,y \in \rho_{\theta}^{-1}(I) \text{ for some } I \in \cG_{\theta}. \end{displaymath}
One also defines the energy $\cE$ almost as before by
\begin{displaymath} \cE := \int_{E_{\delta}} \mu \times \mu(\{(x,y) : x \sim_{\theta} y\}) \, d\theta = \iint |\{\theta \in E_{\delta} : x \sim_{\theta} y\}| \, d\mu x \, d\mu y. \end{displaymath}
The only difference with the earlier notion is that the domain of the $\theta$-integration is restricted to $E_{\delta}$. Following the argument leading to \eqref{form3}, one obtains the familiar lower bound
\begin{equation}\label{form12} \cE \gtrsim \delta^{\sigma} \cdot \left(\log \left(\frac{1}{\delta}\right) \right)^{-6}. \end{equation}
In order to estimate $\cE$ from above, I record the following universal bound:
\begin{lemma}\label{universal} If $x,y \in \R^{3}$ are distinct points, then
\begin{displaymath} |\{\theta \in [0,2\pi) : x \sim_{\theta} y\}| \lesssim \left(\frac{\delta}{|x - y|}\right)^{1/2} \end{displaymath}
\end{lemma}

\begin{proof} Observe that
\begin{displaymath} \{\theta \in [0,2\pi) : x \sim_{\theta} y\} \subset \{\theta \in [0,2\pi) : |\rho_{\theta}(x - y)| \leq \delta\}. \end{displaymath}
The length of the set on the right hand side can be estimated by studying the function $\theta \mapsto \rho_{\theta}(\xi)$, $\xi \in S^{2}$. The key observation is that this function can have at most second order zeros. The details can be found above \cite[(3.6)]{FO2}.
\end{proof}

Next, the proof deviates a little further from the one of Theorem \ref{main}. One defines the cone
\begin{displaymath} C^{E} := \bigcup_{\theta \in E_{\delta}} b_{\theta}, \end{displaymath} 
where $b_{\theta} = \spa(\gamma(\theta) \times \dot{\gamma}(\theta)) = \spa(\cos \theta, \sin \theta, -1)$, as before. If a difference $x - y$ stays far from $C^{E}$, the universal bound in Lemma \ref{universal} can be improved as follows.
\begin{lemma}\label{improvement} Let $0 \leq \tau < 1$, and assume that $y  - x \notin B(C^{E},\delta^{\tau})$. Then
\begin{displaymath} |\{\theta \in E_{\delta} : x \sim_{\theta} y\}| \lesssim \delta^{1 - \tau}. \end{displaymath}
\end{lemma}

\begin{proof} By definition of $B(C^{E},\delta^{\tau})$, one has $\dist(y - x,b_{\theta}) > \delta^{\tau}$ for all $\theta \in E_{\delta}$. Since $b_{\theta} = \ker \tilde{\pi}_{\theta}$, this implies that $|\tilde{\pi}_{\theta}(y - x)| > \delta^{\tau}$ for $\theta \in E_{\delta}$. Rewriting the inequality,
\begin{displaymath} \left[\left((x - y) \cdot \frac{\gamma(\theta)}{|\gamma(\theta)|}\right)^{2} + \left((x - y) \cdot \frac{\dot{\gamma}(\theta)}{|\dot{\gamma}(\theta)|}\right)^{2}\right]^{1/2} = |\tilde{\pi}_{\theta}(x - y)| > \delta^{\tau}. \end{displaymath}
Since $|\gamma(\theta)|$ and $|\dot{\gamma}(\theta)|$ are both bounded from below on $[0,2\pi)$, one may infer that, for some suitable constant $c > 0$, 
\begin{displaymath} \{\theta \in E_{\delta} : x \sim_{\theta} y\} \subset \{\theta : |(x - y) \cdot \gamma(\theta)| > c\delta^{\tau}\} \cup \{\theta : |(x - y) \cdot \dot{\gamma}(\theta)| > c\delta^{\tau}\}. \end{displaymath}
On the other hand, the condition $x \sim_{\theta} y$ always implies that $|(x - y) \cdot \gamma(\theta)| \leq \delta$, so, if $\delta > 0$ is small,
\begin{displaymath} \{\theta \in E_{\delta} : x \sim_{\theta} y\} \subset \{\theta \in [0,2\pi) : |(x - y) \cdot \gamma(\theta)| \leq \delta \text{ and } |(x - y) \cdot \dot{\gamma}(\theta)| > c\delta^{\tau}\}. \end{displaymath}
As long as $x \neq y$, the mapping $\theta \mapsto (x - y) \cdot \gamma(\theta) = \rho_{\theta}(x - y)$ has at most two zeroes on $[0,2\pi)$, and the set $\{\theta : |\rho_{\theta}(x - y)| \leq \delta\}$ is contained in the union of certain intervals around these zeroes. The upper bound on $|(x - y) \cdot \gamma(\theta)|$ and the lower bound on $|(x - y) \cdot \dot{\gamma}(\theta)|$ show that these individual intervals have length $\lesssim \delta^{1 - \tau}$, and the proof of the lemma is complete. \end{proof}

The next goal is to find three points $x_{1},x_{2},x_{3} \in B(0,1)$ such that $|x_{i} - x_{j}| \geq \delta^{13\kappa}$ for $1 \leq i < j \leq 3$ and 
\begin{equation}\label{form14} \mu([x_{1} + B(C^{E},\delta^{\tau})] \cap [x_{2} + B(C^{E},\delta^{\tau})] \cap [x_{3} + B(C^{E},\delta^{\tau})]) \geq \delta^{13\kappa}. \end{equation}
As long as one is not interested in optimising the constants in Theorem \ref{main2}, the number $\tau$ can be chosen freely on the open interval $(0,1/2)$; the value of $\kappa > 0$ will be fixed later, and it will have to be small relative to $\tau$. To reach \eqref{form14}, one -- almost as before -- defines the set $G$ by
\begin{displaymath} G := \{y \in \R^{3} : \mu(y + B(C^{E},\delta^{\tau})) \geq \delta^{\kappa}\}. \end{displaymath}
Write $\cE = I_{G} + I_{\R^{3} \setminus G}$, where
\begin{displaymath} I_{G} = \int_{G} \int |\{\theta \in E_{\delta} : x \sim_{\theta} y\}| \, d\mu x \, d\mu y \end{displaymath}
and
\begin{displaymath} I_{\R^{3} \setminus G} = \int_{\R^{3} \setminus G} \int |\{\theta \in E_{\delta} : x \sim_{\theta} y\}| \, d\mu x \, d\mu y. \end{displaymath}
The part $I_{G}$ is estimated using the universal bound from Lemma \ref{universal}:
\begin{displaymath} I_{G} \lesssim \delta^{1/2} \cdot \int_{G} \int \frac{1}{|x - y|^{1/2}} \, d\mu x \, d\mu y \lesssim_{s} \delta^{1/2} \cdot \mu(G). \end{displaymath}
In the latter inequality one needs the growth condition $\mu(B(x,r)) \lesssim \min\{r^{s},1\}$ with some $s > 1/2$ to ensure that $\int |x - y|^{-1/2} \, d\mu x \lesssim_{s} 1$ for $y \in \R^{3}$. To find an upper bound for $I_{\R^{3} \setminus G}$, another splitting of the integration is required:
\begin{displaymath} I_{\R^{3} \setminus G} = \int_{\R^{3} \setminus G} \int_{y + B(C^{E},\delta^{\tau})} \ldots \, d\mu x \, d\mu y + \int_{\R^{3} \setminus G} \int_{\R^{3} \setminus (y + B(C^{E},\delta^{\tau}))} \ldots \, d\mu x \, d\mu y. \end{displaymath}
These terms will be called $I_{\R^{3} \setminus G}^{1}$ and $I_{\R^{3} \setminus G}^{2}$. As regards $I_{\R^{3} \setminus G}^{1}$, the definition of $y \in \R^{3} \setminus G$ means that $\mu(y + B(C^{E},\delta^{\tau})) < \delta^{\kappa}$. Let 
\begin{displaymath} A_{j}(y) := \{x \in \R^{3} : 2^{j} \leq |x - y| \leq 2^{j + 1}\}. \end{displaymath}
Combining the universal bound from Lemma \ref{universal} with the inequality 
\begin{displaymath} \mu([y + B(C^{E},\delta^{\tau})] \cap A_{j}(y)) \lesssim \min\{\delta^{\kappa},2^{js}\} \leq \delta^{\kappa(1 - 1/2s)} \cdot 2^{j/2}, \quad y \in \R^{3} \setminus G, \end{displaymath}
gives
\begin{align*} I_{\R^{3} \setminus G}^{1} & \lesssim \int_{\R^{3} \setminus G} \int_{B(y,\delta)} 2\pi \, d\mu x \, d\mu y\\
&\quad + \int_{\R^{3} \setminus G} \sum_{\delta \leq 2^{j} \leq 1} \int_{(y + B(C^{E},\delta^{\tau})) \cap A_{j}(y)} |\{\theta \in E_{\delta} : x \sim_{\theta} y\}| \, d\mu x \, d\mu y\\
& \lesssim \delta^{s} + \delta^{1/2} \cdot \int_{\R^{3} \setminus G} \sum_{\delta \leq 2^{j} \leq 1} 2^{-j/2} \cdot \mu([y + B(C^{E},\delta^{\tau})] \cap A_{j}(y)) \, d\mu y\\
& \lesssim \delta^{s} + \delta^{1/2 + \kappa(1 - 1/2s)} \cdot \log \left(\frac{1}{\delta}\right). \end{align*} 
In estimating $I_{\R^{3} \setminus G}^{2}$, one only needs to know that $y - x \notin B(C^{E},\delta^{\tau})$ in the inner integration. This enables the use of Lemma \ref{improvement}:
\begin{align*} I_{\R^{3} \setminus G}^{2} \lesssim \int_{\R^{3}} \int_{\R^{3}  \setminus (y + B(C^{E},\delta^{\tau}))} \delta^{1 - \tau} \, d\mu x \, d\mu y \leq \delta^{1 - \tau}. \end{align*} 
Collecting the three-part upper estimate for $\cE$ and comparing it with the lower bound \eqref{form12} yields
\begin{displaymath} \delta^{\sigma} \cdot \left(\log \left(\frac{1}{\delta} \right) \right)^{-6} \lesssim \cE \lesssim \delta^{1/2} \cdot \mu(G) + \delta^{s} + \delta^{1/2 + \kappa(1 - 1/2s)} \cdot \log \left(\frac{1}{\delta}\right) + \delta^{1 - \tau}. \end{displaymath}
Now, as long as $0 < \kappa,\tau < 1/2$ are fixed parameters, assuming that $\sigma$ is close enough to $1/2$ shows that the sum of the three last terms on the right hand side cannot dominate the left hand side for small $\delta$. Thus, one obtains
\begin{equation}\label{form13} \mu(G) \gtrsim \delta^{\sigma - 1/2} \cdot \left(\log \left(\frac{1}{\delta} \right) \right)^{-6} \geq \delta^{\kappa}, \end{equation} 
where the second inequality is, once again, reached simply by taking $\delta > 0$ small and $\sigma$ close to $1/2$. Next, an application of Hölder's inequality similar to the one seen in the proof of Theorem \ref{main} gives
\begin{displaymath} A := \iiint \mu([x + B(C^{E},\delta^{\tau})] \cap [y + B(C^{E},\delta^{\tau})] \cap [z + B(C^{E},\delta^{\tau})]) \, d\mu x \, d\mu y \, d\mu z \gtrsim \delta^{6\kappa}. \end{displaymath} 
Recall that the aim is to find a triple $x_{1},x_{2},x_{3} \in \spt \mu \subset B(0,1)$ such \eqref{form14} holds and the mutual distance of the points $x_{i}$ is at least $\delta^{13\kappa}$. If this cannot be done, then the condition
\begin{displaymath} \min\{|x_{i} - x_{j}| : 1 \leq i < j \leq 3\} \geq \delta^{13\kappa} \end{displaymath}
implies that 
\begin{displaymath} \mu([(x_{1} + B(C^{E},\delta^{\tau})] \cap [x_{2} + B(C^{E},\delta^{\tau})] \cap [x_{3} + B(C^{E},\delta^{\tau})]) < \delta^{13\kappa} \end{displaymath}
for all $x_{1},x_{2},x_{3} \in \spt \mu$. Thus, one finds that
\begin{align*} A & \leq \sum_{1 \leq i_{1} \leq i_{2} \leq i_{3} \leq 3} \iint \int_{B(x_{i_{1}},\delta^{13\kappa}) \cup B(x_{i_{2}},\delta^{13\kappa})} \, d\mu x_{i_{3}} \, d\mu x_{i_{1}} \, d\mu x_{i_{2}}\\
& \quad + \iiint_{\{\min\{|x_{i} - x_{j}| : 1 \leq i < j \leq 3\} \geq \delta^{13\kappa}\}} \delta^{13\kappa} \, d\mu x_{1} \, d\mu x_{2} \, d\mu x_{3} \lesssim \delta^{13\kappa s} + \delta^{13\kappa}. \end{align*}
Since $s > 1/2$, for small enough $\delta > 0$ this violates the lower for $A$ obtained above. Thus, there must exist points $x_{1},x_{2},x_{3} \in B(0,1)$ such that $|x_{i} - x_{j}| \geq \delta^{13\kappa}$ and \eqref{form14} holds. Without loss of generality, assume that $x_{1} = 0$.

Now, it is again time to introduce the relevant geometric lemma:
\begin{lemma}[Three cones lemma]\label{threeCones} There is an absolute constant $c \in (0,1)$ such that the following holds for small enough $\delta > 0$. Let $C = \{(x,y,z) : x^{2} + y^{2} = z^{2}\}$, and let $p,q \in B(0,1)$ be points satisfying
\begin{displaymath} \min\{|p|,|q|,|p - q|\} \geq \delta^{c}. \end{displaymath}
Write
\begin{displaymath} \cC_{0} := B(C,\delta), \quad \cC_{p} := p + \cC_{0}, \quad \cC_{q} := q + \cC_{0}. \end{displaymath}
Then the intersection
\begin{displaymath} (\cC_{0} \cap \cC_{p} \cap \cC_{q}) \cap B(0,1) \end{displaymath}
is contained in the $\delta^{c}$-neighbourhood of at most two of the lines on $C$. \end{lemma}

Assuming that $13\kappa/\tau < c$ and applying the three cones lemma with $p = x_{2}, q = x_{3}$, and with $\delta^{\tau}$ in place of $\delta$, one finds that the intersection
\begin{displaymath} (x_{1} + B(C^{E},\delta^{\tau})) \cap (x_{2} + B(C^{E},\delta^{\tau})) \cap (x_{3} + B(C^{E},\delta^{\tau})) \cap B(0,1) \end{displaymath}
is contained in the $\delta^{c \tau}$-neighbourhood of at most two lines on $C$. Let $L_{1},L_{2} \subset C$ be these lines. It follows from \eqref{form14} that either $\mu(B(C^{E},\delta^{\tau}) \cap B(L_{1},\delta^{c \tau})) \gtrsim \delta^{13\kappa}$ or $\mu(B(C^{E},\delta^{\tau}) \cap B(L_{2},\delta^{c \tau})) \gtrsim \delta^{13\kappa}$; assume that the former options holds. Then also
\begin{equation}\label{form15} \mu(B(C^{E},\delta^{c\tau}) \cap B(L_{1},\delta^{c\tau})) \gtrsim \delta^{13\kappa}, \end{equation}
by monotonicity and $c < 1$. There are two options: either $L_{1}$ forms a large angle with all the lines on $b_{\theta} \subset C^{E}$, or $L_{1}$ forms a small angle with a certain line on $C^{E}$. More precisely, assume first that the angle between $L_{1}$ and each line $b_{\theta} \subset C^{E}$, $\theta \in E$, is at least $\delta^{c \tau/2}$. Then, since $L_{1}$ intersects all the lines on $C^{E}$ at the origin, simple geometry (as in \cite[(4)]{Wo}) shows that
\begin{displaymath} B(C^{E},\delta^{c\tau}) \cap B(L_{1},\delta^{c\tau}) \subset B(0,\delta^{c\tau/3}) \end{displaymath}
for $\delta > 0$ small enough. However, this would imply that 
\begin{displaymath} \mu(B(C^{E},\delta^{c\tau}) \cap B(L_{1},\delta^{c\tau})) \lesssim \delta^{c s \tau/3}, \end{displaymath}
which, using \eqref{form15}, can be ruled out by choosing $\kappa > 0$ small enough to begin with. The conclusion is that there exists a line $L = b_{\theta} \subset C^{E}$ such that the angle between $L_{1}$ and $L$ is smaller than $\delta^{c\tau/2}$. It follows that $B(L_{1},\delta^{c\tau}) \cap B(0,1) \subset B(L,\delta^{c\tau/3})$ for small enough $\delta > 0$, and so \eqref{form15} yields
\begin{displaymath} \mu(B(L,\delta^{c\tau/3})) \gtrsim \delta^{13\kappa}. \end{displaymath}
To complete the proof of the theorem, apply Lemma \ref{littleMass} to the projected measure $\tilde{\pi}_{\theta\sharp}\mu$, where $L = b_{\theta}$. Since $\theta \in E$, one has \eqref{form10}, and then Lemma \ref{littleMass} yields an upper bound for the $\mu$ mass of the pre-images of discs on $\tilde{V}_{\theta}$. The neighbourhood $B(L,\delta^{c \tau/3})$ is such a pre-image, so
\begin{displaymath} \mu(B(L,\delta^{c\tau/3})) \lesssim (\delta^{c\tau/3})^{s/2} \sim \delta^{c s\tau/6}. \end{displaymath}
Choosing $\kappa < c s \tau/78$, this contradicts the lower bound from \eqref{form15} and completes the proof of Theorem \ref{main2}.
\end{proof}

\appendix

\section{Proof of the two cones lemma}\label{AppendixB}

Recall that $\gamma$ was an $S^{2}$-valued $\cC^{3}$-curve initially defined on an open subinterval of $\R$, which we now choose to denote by $J_{0}$, such that 
\begin{displaymath} \spa\{\gamma(\theta),\dot{\gamma}(\theta),\ddot{\gamma}(\theta)\} = \R^{3}, \qquad \theta \in J_{0}. \end{displaymath}
We reiterate the statement of the Two cones lemma:

\begin{lemma}[Two cones lemma]\label{twoCones} Let $\gamma$ be as above. Then, the following holds for small enough $\epsilon > 0$, for all short enough subintervals $J \subset J_{0}$ (see Remark \ref{parametrisationRemark} below), for small enough $\delta > 0$, and for $5\epsilon \leq \tau < 1/2$. Let 
\begin{displaymath} \cC := \bigcup_{\theta \in J} B(\ell_{\theta},\delta),  \end{displaymath} 
where $\ell_{\theta}$ is the \textbf{half-line} $\ell_{\theta} = \{r\gamma(\theta) : r \geq 0\}$, and assume that $p \in \R^{3}$ is a point with $|p| \geq \delta^{\epsilon}$. Then the intersection
\begin{displaymath} \cC \cap (\cC + p) \cap B(0,1) \end{displaymath}
can be covered by two balls of diameter $\lesssim \delta^{\epsilon}$, plus either 
\begin{itemize}
\item[(a)] $\lessapprox \delta^{-1/2 - 2\tau}$ balls of diameter $\lessapprox \delta^{1/2}$, or
\item[(b)]  $\lessapprox \delta^{-\tau/4}$ balls of diameter $\lessapprox \delta^{\tau/4}$. 
\end{itemize}
\end{lemma}

\begin{remark}\label{parametrisationRemark} The correct interpretation of the lemma is that one of the options (a) or (b) always holds, depending on $p$, and not that one can choose freely between them. The notation $A \lessapprox B$ means that $A \leq R\delta^{-R\epsilon}B$ for some absolute constant $R \geq 1$, where $\epsilon > 0$ is the constant from the lemma. Writing $A \gtrapprox B$ means that $B \lessapprox A$ (that is, $A \geq (1/R)\delta^{R\epsilon}B$). It will be made apparent after \eqref{transversality} below, how small is a "small enough $\epsilon > 0$", but the meaning of "short enough $J$" will be explained right now. First of all, a precise formulation of the phrase would read as follows: one can pick any point $\theta_{0}$ on the interval $J_{0}$, where $\gamma$ was originally defined, and then restrict this interval to a neighbourhood $J$ of $\theta_{0}$, so that the lemma holds for $J$, and the length requirements for $J$ depend only on $\gamma$ and $\theta_{0}$.

Let $C$ be the surface
\begin{displaymath} C := \bigcup_{\theta \in J} \ell_{\theta}. \end{displaymath}
For convenience, assume that $J \subset J_{0}$ is closed. It is desirable to be able to parametrise $C$ as
\begin{equation}\label{parametrisation} C = \left\{\left(t,hf\left(\frac{t}{h}\right),h\right) : h \geq 0, t \in h I\right\}, \end{equation}
where $I \subset \R$ is a compact interval, and $f \colon I \to \R$ is a smooth Lipschitz function satisfying 
\begin{displaymath} f'' \geq \eta > 0. \end{displaymath}
This can be done, if $J$ is "short enough". To understand the restrictions, assume that $\theta_{0} \in J$, and -- without loss of generality -- $\gamma(\theta_{0}) = (0,0,1) \in S^{2}$. Then, the tangent plane of $S^{2}$ at $\gamma(\theta_{0})$ is $H = \{(x,y,1) : x,y \in \R\}$, and, if $J$ is so short that $\gamma(J)$ lies in the well inside the upper hemisphere of $S^{2}$, one can define a path $\lambda \colon J \to H$ by
\begin{displaymath} \lambda(\theta) := \left(\frac{\gamma_{1}(\theta)}{\gamma_{3}(\theta)},\frac{\gamma_{2}(\theta)}{\gamma_{3}(\theta)},1\right). \end{displaymath}
Then, it is clear that
\begin{displaymath} C = \bigcup_{\theta \in \tilde{J}} \spa(\lambda(\theta)),  \end{displaymath}
where $\spa(\lambda(\theta))$ refers to the half-line spanned by $\lambda(\theta)$. Moreover, since $\gamma = \gamma_{3}\lambda$, one has the following relations for the derivatives:
\begin{displaymath} \gamma_{3}'\lambda + \lambda'\gamma_{3} = \gamma' \quad \text{and} \quad \gamma_{3}''\lambda + 2\gamma_{3}'\lambda' + \gamma_{3}\lambda'' = \gamma''. \end{displaymath}
This leads to 
\begin{align*} \gamma'' \cdot (\gamma \times \gamma') = (\gamma_{3}''\lambda + 2\gamma_{3}'\lambda' + \gamma_{3}\lambda'') \cdot (\gamma_{3}\lambda \times [\gamma_{3}'\lambda + \lambda'\gamma_{3}]) = \gamma_{3}^{3}\lambda'' \cdot (\lambda \times \lambda'). \end{align*} 
Since $\spa(\{\gamma,\gamma',\gamma''\}) = \R^{3}$ implies that $\gamma'' \cdot (\gamma \times \gamma') \neq 0$, the relation above shows that $\lambda''(\theta) \neq 0$ for $\theta \in J$. Moreover, $|\lambda''(\theta)| \geq \tilde{\eta} > 0$ for $\theta \in J$ by compactness.  Using this, a routine argument shows that $\lambda(J)$ can be parametrised as
\begin{equation}\label{parametrisation2} \lambda(J) = \{(t,f(t),1) : t \in I\}, \end{equation}
where $f \colon I \to \R$ is a Lipschitz function with $|f''(t)| \geq \eta > 0$ for $t \in I$, and $I \subset \R$ is a compact interval (this may involve a rotation of coordinates by $90$ degrees and making $J$ a little shorter around $\theta_{0}$, if one is so unlucky that $\lambda'(\theta_{0})$ is parallel to the $y$-axis). Then, without loss of generality, one may assume that $f''$ is positive on $I$, and $I = [0,1]$. Finally, \eqref{parametrisation2} implies \eqref{parametrisation}, since 
\begin{displaymath} C \cap \{(x,y,h) : x,y \in \R\} = h\lambda(J) = \left\{\left(t,hf\left(\frac{t}{h}\right),h\right) : 0 \leq t \leq h\right\}, \quad h \geq 0. \end{displaymath}
So, the parametrisation \eqref{parametrisation} is possible, once $J$ is "short enough". This hypothesis will be also be needed in the proof below -- mainly in the form that $\gamma(J)$ is contained in the upper hemisphere -- but I will make no further mention about it.
\end{remark}

Now that the assumptions and notations have been clarified, I start the preparations for the actual proof. Assuming that $C$ is parametrised as in \eqref{parametrisation}, with $I = [0,1]$, any translate of $C$ can be written as
\begin{displaymath} C + p = \left\{\left(t, (h + a)f \left(\frac{t + b}{h + a}\right) + c, h \right) : t \in \R, h \geq -a, 0 \leq t + b \leq h + a \right\}, \end{displaymath}
where $p = (-b,c,-a)$. To see this, note that if $x \in C + (-b,c,-a)$, then $x = (t,hf(t/h),h) + (-b,c,-a)$ for some $0 \leq t \leq h$. Then, writing $h' := h - a$ and $t' := t - b$, one has $h' \geq -a$ and $0 \leq t' + b \leq h' + a$, and $x = (t',(h' + a)f((t' + b)/(h' + a)) + c, h')$.

With the identification $p \cong (-a,b,-c)$ as above, the assumption $\delta^{\epsilon} \leq |p| \leq 1$ translates to
\begin{equation}\label{sizeabc} \delta^{\epsilon} \leq \max\{|a|,|b|,|c|\} \leq 1. \end{equation}

Before getting anywhere, one also needs to declare that
\begin{equation}\label{heightAssumption} \min\{h,h + a\} \geq \delta^{\epsilon} \end{equation}
for all the heights $h$, which one encounters below. Indeed, the "two balls of diameter $\lesssim \delta^{\epsilon}$" appearing in the statement of the lemma are used to cover the sets
\begin{displaymath} \cC \cap \R^{2} \times [-\delta^{\epsilon},\delta^{\epsilon}] \quad \text{and} \quad (\cC + p) \cap \R^{2} \times [- a - \delta^{\epsilon}, - a + \delta^{\epsilon}]. \end{displaymath}
After this, all the points $(t,y,h) \in \cC \cap (\cC + p)$ where \eqref{heightAssumption} fails have already been covered, and one can assume \eqref{heightAssumption} in the sequel. The upshot is that the functions 
\begin{equation}\label{functions} t \mapsto hf\left(\frac{t}{h}\right) \quad \text{and} \quad t \mapsto (h + a)f\left(\frac{t + b}{h + a}\right) \end{equation}
and their difference are $L$-Lipschitz with $L \lessapprox 1$ under the assumption \eqref{heightAssumption}. 

Next, write
\begin{displaymath} H(h,r) := \R^{2} \times [h - r, h + r] \end{displaymath}
for the horizontal slab of width $2r$, with vertical centre at $h$. For $r = 0$, this is abbreviated to $H(h) := H(h,0)$. 

\subsection{Overview of the proof} I will now explain the structure of the proof at a semi-technical level, introducing notation as I go. There are two main steps. The first is to restrict the intersection $\cC \cap (\cC + p) \cap B(0,1)$ to some fixed height $h$ satisfying \eqref{heightAssumption}, and to study a single slice of the form
\begin{displaymath} H(h) \cap \cC \cap (\cC + p), \qquad h \in [-1,1]. \end{displaymath}
The main analytic tool in this task is the function
\begin{displaymath} d_{h}(t) = hf\left(\frac{t}{h}\right) - \left[(h + a)f\left(\frac{t + b}{h + a}\right) + c \right], \end{displaymath}
defined for
\begin{displaymath} t \in I_{h} := [\max\{0,-b\},\min\{h,h + a - b\}]. \end{displaymath}
So, $I_{h}$ is simply the intersection of the domains of definition of the functions in \eqref{functions}. See Figure \ref{fig1} for the graphical interpretation and recall that $d_{h}$ is Lipschitz with constant $\lessapprox 1$ under the hypothesis \eqref{heightAssumption}.
\begin{figure}[h!]
\begin{center}
\includegraphics[scale = 0.5]{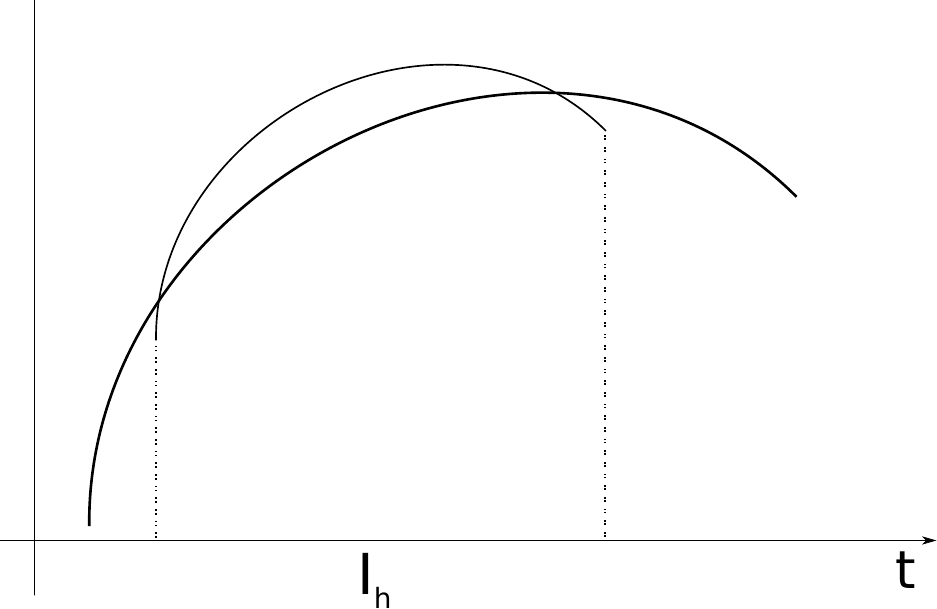}
\caption{The function $d_{h}$ measures the difference between the heights of the graphs on the interval $I_{h}$, where both graphs are well-defined.}\label{fig1}
\end{center}
\end{figure}
Now, in the first step of the proof, one is trying to establish that the set $H(h) \cap \cC \cap (\cC + p)$ (interpreted as a subset of $\R^{2}$) can be covered by at most two small discs. To do this, one observes that
\begin{displaymath} H(h) \cap \cC \cap (\cC + p) \subset B(\Gamma_{1},\lessapprox \delta) \cap B(\Gamma_{2},\lessapprox \delta), \end{displaymath}
where $B(\Gamma_{1},\lessapprox \delta)$ and $B(\Gamma_{2},\lessapprox \delta)$ stand for the $A$-neighbourhoods of the graphs of the functions
\begin{displaymath} g_{1}(t) = hf\left(\frac{t}{h}\right) \quad \text{and} \quad g_{2}(t) = (h + a)f\left(\frac{t + b}{h + a}\right). \end{displaymath}
for some $A \lessapprox \delta$. With this notation, $d_{h} = g_{1} - g_{2}$, and the proof will proceed by establishing that 
\begin{displaymath} \{t \in I_{h} : |d_{h}(t)| \lessapprox \delta\} \end{displaymath}
can be covered by two small intervals \textbf{with midpoints located either at the zeros of $d_{h}$, or at the endpoints of the interval $I_{h}$}. To obtain from this information the desired disc-cover for $B(\Gamma_{1},\lessapprox \delta) \cap B(\Gamma_{2},\lessapprox \delta)$, one uses the following simple fact:
\begin{fact}\label{comparison} Assume that $g_{i} \colon I_{i} \to \R$, $i \in \{1,2\}$ are two $L$-Lipschitz functions defined on the intervals $I_{1},I_{2} \subset \R$, where $I_{1} \cap I_{2} \neq \emptyset$ and $L \geq 1$. Let $\Gamma_{i} \subset \R^{2}$ be the graph of $g_{i}$,
\begin{displaymath} \Gamma_{i} = \{(t,g_{i}(t)) : t \in I_{i}\}. \end{displaymath}
Then $B(\Gamma_{1},\delta) \cap B(\Gamma_{2},\delta)$ is contained in the $6L\delta$-neighbourhood of the set 
\begin{displaymath} g_{1}(\{t \in I_{1} \cap I_{2} : |(g_{1} - g_{2})(t)| \leq 6L\delta\}). \end{displaymath}
\end{fact} 

\begin{proof} Repeated application of the triangle inequality. \end{proof} 

In the present application, $I_{1} \cap I_{2} = I_{h}$, so -- to apply the fact -- one should make sure that $I_{h} \neq \emptyset$. In general, the set $H(\neq\emptyset) := \{h \in [-1,1] : I_{h} \neq \emptyset\}$ is a closed subinterval of $[-1,1]$, by inspecting the definition of $I_{h}$. If it is a \textbf{strict} subinterval, one should restrict all further attention to $H(\neq\emptyset)$. To avoid introducing any further notation, however, I will assume that $H(\neq \emptyset) = [-1,1]$.

So, once it has been shown that $\{t \in I_{h} : |d_{h}(t)| \lessapprox \delta\}$ can be covered by two intervals $I_{1}^{h},I_{2}^{h}$ of length $\leq \delta^{\beta}$, for some $\beta > 0$, Fact \ref{comparison} shows that the intersection $\Gamma_{1}(\lessapprox \delta) \cap \Gamma_{2}(\lessapprox \delta) \supset H(h) \cap \cC \cap (\cC + p)$ can be covered by two discs of diameter $\lessapprox \delta^{\beta}$. More precisely, the centres of the discs can be chosen to be of the form 
\begin{displaymath} (t,g_{1}(t),h) = \left(t,hf\left(\frac{t}{h}\right),h\right), \end{displaymath}
where either 
\begin{displaymath} d_{h}(t) = 0 \quad \text{or} \quad t \in \partial I_{h}, \end{displaymath}
as long as this holds for the midpoints $t$ of the intervals $I_{1}^{h}$ and $I_{2}^{h}$.

The second main step of the proof is "gluing together" the slices $H(h) \cap \cC \cap (\cC + p)$ for various $h \in [-1,1]$. As there is only an "abstract" statement that each $h$-slice can be covered by two small discs, there remains a risk of the -- admittedly unbelievable -- situation that the centres of the discs vary so much for different $h$ that the union of the slices can no longer be covered by a small number of small balls. Morally, the solution is to parametrise the centres of the discs by a Lipschitz function with Lipschitz constant $\lessapprox 1$. Since the centres were connected with the zeros of $d_{h}$, this sounds like a job for the implicit function theorem (IFT).

A straightforward application of the IFT runs soon into trouble, and it is instructive to see why. I will explain the "argument". One first defines a function $d \colon \R \times \R \to \R$ by $d(t,h) = d_{h}(t)$. Then, as remarked above, the midpoints of the at most two intervals covering $\{t : |d_{h}(t)| \lessapprox \delta\}$ are situated at the zeros of the function $d_{h}(t)$ (or at the endpoints of $I_{h}$, but ignore this possibility for now). So, one can start off at some $(t_{0},h_{0})$ such that $d(t_{0},h_{0}) = 0$ and try to apply the IFT: if everything works out, the theorem pops out a smooth function $\psi$ of the variable $h$ such that $d(\psi(h),h) = 0$ for $h$ close enough to $h_{0}$. Then, because $d_{h}$ can have at most two zeros on $I_{h}$ (easy), and since $d_{h}(\psi(h)) = 0$, it \textbf{has} to be the case that $\psi(h)$ is among the midpoints of the "abstractly" chosen two intervals covering $\{t : |d_{h}(t)| \leq \delta\}$. So, this strategy might conceivably produce a smooth parametrisation for the midpoints. As a corollary, one would obtain a smooth parametrisation for the centres of the discs, since -- as discussed above -- these can be taken to be of the form $(\psi(h),hf(\psi(h)/h),h)$. At this point, the proof would practically be finished.

There are two issues. First, the IFT gives no indication of the size of the interval around $h_{0}$ such that $g(h)$ is well-defined. However, one essentially needs a global parametrisation here. Second, the principal hypothesis of the IFT in this situation is that $d'_{h_{0}}(t_{0}) \neq 0$, and this can easily fail, if the parameters $a,b,c$ are chosen suitably. Such an event is depicted in Figure \ref{fig2}.
\begin{figure}[h!]
\begin{center}
\includegraphics[scale = 0.5]{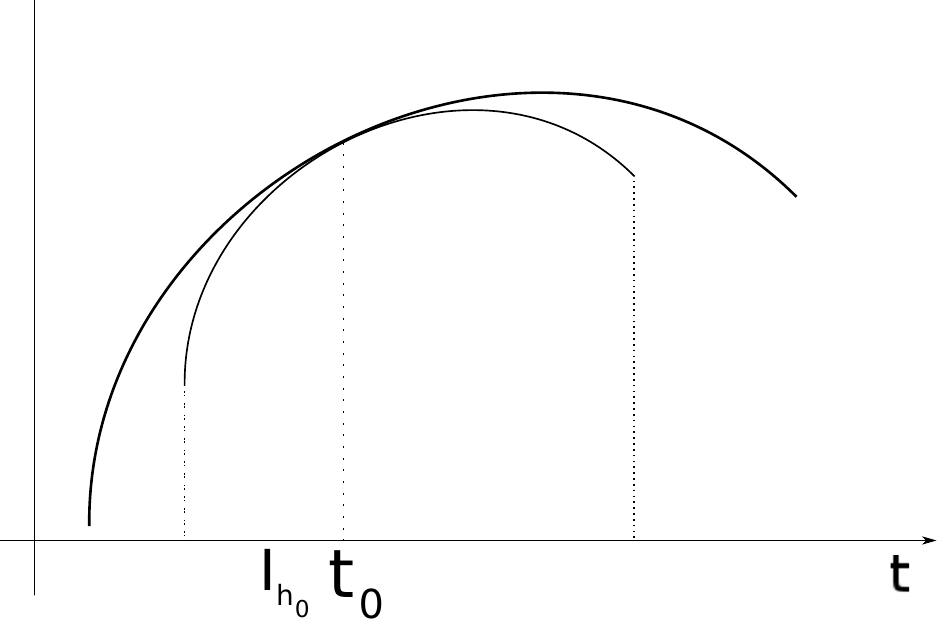}
\caption{It can happen that $d_{h_{0}}(t_{0}) = 0 = d'_{h_{0}}(t_{0})$.}\label{fig2}
\end{center}
\end{figure}
The second issue would kill the approach, were it not the case that the situation of Figure \ref{fig2} can be excluded \emph{a priori}. In fact, if $d_{h_{0}}(t_{0}) \sim 0 \sim d'_{h_{0}}(t_{0})$ for some $(t_{0},h_{0})$, one can extract an algebraic relation between the parameters $a,b,c$ and use it to finish off the whole proof in an \emph{ad hoc} manner. This leads to alternative (b) in the lemma. After the bad case has been excluded, one can prove a global "poor man's version" of the implicit function theorem by hand, and conclude the proof along the lines discussed above.
\subsection{The details} According to the proof outline above, the first task is to dispose of the situation, where $d_{h_{0}}(t_{0}) \sim 0 \sim d'_{h_{0}}(t_{0})$ for some $(t_{0},h_{0})$. This is the content of the following proposition:

\begin{proposition}\label{secondAlternative} Let $5\epsilon \leq \tau < 1/2$ where $\epsilon > 0$ is the constant from \eqref{sizeabc}. Assume that there exists a height $h_{0}$ satisfying \eqref{heightAssumption}, and a point $t_{0} \in I_{h_{0}}$ such that
\begin{displaymath} |d_{h_{0}}(t_{0})| \leq \delta^{\tau} \quad \text{and} \quad |d'_{h_{0}}(t_{0})| \leq \delta^{\tau}. \end{displaymath}
Then, the intersection $\cC \cap (\cC + p) \cap B(0,1)$ can be covered by two balls of diameter $\lesssim \delta^{\epsilon}$, plus $\lessapprox \delta^{-\tau/4}$ balls of diameter $\lessapprox \delta^{\tau/4}$.
\end{proposition}

\begin{proof} The derivative of $d_{h}$ has the relatively simple expression
\begin{equation}\label{derivativeOfD} d_{h}'(t) = f'\left(\frac{t}{h}\right) - f'\left(\frac{t + b}{h + a}\right). \end{equation}
Now, recall that $f'' \geq \eta > 0$. In particular, if $|d'_{h_{0}}(t_{0})| \leq \delta^{\tau}$, it follows that
\begin{equation}\label{form16} \left|\frac{t_{0}}{h_{0}} - \frac{t_{0} + b}{h_{0} + a}\right| \lesssim \left|f'\left(\frac{t_{0}}{h_{0}}\right) - f'\left(\frac{t_{0} + b}{h_{0} + a}\right)\right| = |d_{h_{0}}'(t_{0})| \leq \delta^{\tau}. \end{equation}
Consequently,
\begin{displaymath} \left|\frac{t_{0}}{h_{0}} - \frac{b}{a}\right| = \left|\frac{t_{0}}{h_{0}} - \frac{t_{0} + b}{h_{0} + a}\right| \cdot \frac{|h_{0} + a|}{|a|} \lesssim \frac{|h_{0} + a|}{|a|} \cdot \delta^{\tau} \lesssim \frac{\delta^{\tau}}{|a|}, \end{displaymath}
and also
\begin{equation}\label{form25} \left|\frac{b}{a} - \frac{t_{0} + b}{h_{0} + a} \right| \lesssim \frac{\delta^{\tau}}{|a|} + \delta^{\tau} \lesssim \frac{\delta^{\tau}}{|a|}. \end{equation}
Next, using the fact that $f$ is Lipschitz, one deduces from \eqref{form16} that
\begin{displaymath} \left|h_{0}f\left(\frac{t_{0}}{h_{0}}\right) - h_{0}f\left(\frac{t_{0} + b}{h_{0} + a}\right) \right| \lesssim \delta^{\tau}, \end{displaymath}
so that (by the triangle inequality and the definition of $d_{h_{0}}$)
\begin{equation}\label{form18} \left|a f\left(\frac{t_{0} + b}{h_{0} + a}\right) + c \right| \lesssim \delta^{\tau} + |d_{h_{0}}(t_{0})| \lesssim \delta^{\tau}. \end{equation}
Now, if $b/a \in [0,1]$ (so that $f(b/a)$ is well-defined), one can argue as follows (I will come back to this simplifying assumption later). Combining \eqref{form25} and \eqref{form18}, and using the Lipschitz property of $f$,
\begin{equation}\label{algRelation} \left|af\left(\frac{b}{a}\right) + c \right| \lesssim \frac{\delta^{\tau}}{|a|}. \end{equation}
This is not very useful, if $|a|$ is small, say $|a| \leq \delta^{\tau/4}$. However, since $\tau/4 > \epsilon$, the condition that $|a| \leq \delta^{\tau/4}$ forces $|c| \geq \delta^{\epsilon}$ or $|b| \geq \delta^{\epsilon}$ by \eqref{sizeabc}. But if $|a| \leq \delta^{\tau/4}$, then also $|c| \lesssim \delta^{\tau/4}$ by \eqref{form18}, so it has to be the case that $|b| \geq \delta^{\epsilon}$. In this case
\begin{displaymath} \left|\frac{t_{0}}{h_{0}} - \frac{t_{0} + b}{h_{0} + a}\right| = \left|\frac{bh_{0} - at_{0}}{h_{0}(h_{0} + a)}\right| \geq \frac{|b|}{h_{0} + a} - \frac{|a|t_{0}}{h_{0}(h_{0} + a)}. \end{displaymath}
Since $t_{0}/h_{0} \leq 1$ -- by $t_{0} \in I_{h_{0}}$ -- the last expression is further bounded from below by $(|b| - |a|)/(h_{0} + a) \gtrsim \delta^{\epsilon}$, which is a contradiction in light of \eqref{form16}. The conclusion is that $|a| \geq \delta^{\tau/4}$ under the hypotheses of the lemma. Then, \eqref{algRelation} gives
\begin{equation}\label{form19} \left|d_{h}\left(\frac{h b}{a}\right) \right| = \left|af\left(\frac{b}{a}\right) + c \right| \leq C\delta^{3\tau/4} \end{equation}
for \textbf{every} $h \in [-1,1]$, and not just $h = h_{0}$ (the first equality in \eqref{form19} being simply the definition of $d_{h}$). This will have the consequence that the set $\{t : |d_{h}(t)| \leq \delta\}$ is contained in a single (short) interval around $t(h) = hb/a$. To see why, one has to show that $|d_{h}(t)|$ is large, when $|t - hb/a|$ is large. Assume, for example, that $t > hb/a$. Now, since the only zero of $d_{h}'(t)$ is at $t = hb/a$, the function $d_{h}'$ has constant sign on the interval $[hb/a,t]$. This sign could be determined from $a$ and $b$, but it does not affect the computations; I will simply assume that it is positive. So, using $f'' \geq \eta$ again,
\begin{align*} |d_{h}(t)| & \geq \left|d_{h}(t) - d_{h}\left(\frac{hb}{a}\right)\right| - C\delta^{3\tau/4} = \int_{hb/a}^{t} d'_{h}(r) \, dr - C\delta^{3\tau/4}\\
& \gtrsim \int_{hb/a}^{t} \left[\frac{r}{h} - \frac{r + b}{h + a} \right] \, dr - C\delta^{3\tau/4} = \int_{hb/a}^{t} \frac{ra - hb}{h(h + a)} \, dr - C\delta^{3\tau/4}\\
& = \int_{hb/a}^{t} \frac{(r - hb/a)a}{h(h + a)} \, dr - C\delta^{3\tau/4} \gtrsim \delta^{\tau/4}(t - hb/a)^{2} - C\delta^{3\tau/4}. \end{align*} 
This is far larger than $\delta^{3\tau/4}$, as soon as $\delta^{\tau/4}(t - hb/a)^{2} \geq 2C\delta^{3\tau/4}$, which happens as soon as $(t - hb/a) \geq \sqrt{2C}\delta^{\tau/4}$. So, since $|d_{h}(t)| \lessapprox \delta$ implies that $|d_{h}(t)| \leq \delta^{3\tau/4}$, this gives
\begin{equation}\label{form24} \{t \in I_{h} : |d_{h}(t)| \lessapprox \delta\} \subset [hb/a - c\delta^{\tau/4}, hb/a + c\delta^{\tau/4}] \end{equation}
for some large enough constant $c > 0$. So, the sets $\{t \in I_{h} : |d_{h}| \lessapprox \delta\}$, $\min\{h,h + a\} \geq \delta^{\epsilon}$, can be covered by a single short interval each, the midpoint of which depends smoothly on $t$. The rest of the argument follows the outline described earlier. Here are the details once more: using Fact \ref{comparison}, the inclusion \eqref{form24} yields a covering of $H(h) \cap \cC \cap (\cC + p)$ by a single disc of diameter $\lessapprox \delta^{\tau/4}$, centred at
\begin{displaymath} \operatorname{centre}(h) := \left(\frac{hb}{a},hf\left(\frac{b}{a}\right),h\right). \end{displaymath} 
Since $h \mapsto \operatorname{centre}(h)$ is Lipschitz with bounded constants (recalling that $0 \leq b/a \leq 1$), this means that $\cC \cap (\cC + p) \cap B(0,1)$ can be covered by $\lessapprox \delta^{-\tau/4}$ balls of diameter $\lessapprox \delta^{\tau/4}$, and the proof of the proposition is complete.

If $b/a \notin [0,1]$, the details are similar but messier. The extra assumption was not used before \eqref{algRelation}, so \eqref{form16}--\eqref{form18} hold. Also, $|a| \geq \delta^{\tau/4}$, which implies that
\begin{displaymath} \left|\frac{t_{0} + b}{h_{0} + a} - \frac{b}{a} \right| \lesssim \delta^{3\tau/4} \end{displaymath}
by \eqref{form25}. By definition of $I_{h_{0}}$, one has $(t_{0} + b)/(h_{0} + a) \in [0,1]$, so the fact that $b/a \notin [0,1]$ implies that either $0$ or $1$ has to lie between $(t_{0} + b)/(h_{0} + a)$ and $b/a$, at distance $\lesssim \delta^{3\tau/4}$ from both numbers. Assume, for instance, that $1$ has this property, so that $b/a > 1$. Now $t(h) := h + a - b$ will play the role of the special point $h b/a$ above (the reason being that $(t(h) + b)/(h + a) = 1$; if $0$ was picked instead of 1, the choice $t(h) = -b$ would be correct). The claim is that $\{t \in I_{h} : |d_{h}(t)| \lessapprox \delta\}$ is contained in a single short interval centred at $t(h)$. First, note that
\begin{displaymath} d_{h}(t(h)) = hf\left(\frac{h + a - b}{h}\right) - [(h + a)f(1) + c] \end{displaymath}
is well-defined for all $h \in [-1,1]$, since $h + a - b \geq 0$ and $a - b \leq 0$: the first condition is necessary for $I_{h} \neq \emptyset$ (an assumption I made at the beginning), and the second condition is equivalent to $b/a > 1$.

Then, using that $f$ is Lispschitz with bounded constants, combined with the fact that both numbers $b/a$ and $(t_{0} + b)/(h_{0} + a)$ are very close to one, and \eqref{form18}, 
\begin{align*} |d_{h}(t(h))| & \leq |h|\left|f\left(1 + \frac{a - b}{h}\right) - f(1) \right| + |af(1) + c|\\
& \lesssim |a - b| + \left|af\left(\frac{t_{0} + b}{h_{0} + a} \right) + c \right| + \delta^{3\tau/4} \lesssim \delta^{3\tau/4}. \end{align*}
This is the analogue of \eqref{form19}, and the proof can now be concluded in the same spirit as before; one should note $d_{h}'$ has constant sign on the whole interval $I_{h}$, because $d_{h}'$ could only have a zero at $hb/a \notin I_{h}$. I omit the rest of the details. \end{proof}

In the sequel, one is entitled to assume that
\begin{equation}\label{transversality} |d_{h}(t)| \leq \delta^{\tau} \quad \Longrightarrow \quad |d'_{h}(t)| \geq \delta^{\tau}, \end{equation} 
if $5\epsilon \leq \tau < 1/2$, and $t \in I_{h}$.

\begin{proof}[Proof of Lemma \ref{twoCones}] Assume that $a \leq 0$; this corresponds to the case that the vertex of the cone $C + p$ is above the $xy$-plane. The case with $a > 0$ is treated similarly. Recall that
\begin{displaymath} d_{h}'(t) = f'\left(\frac{t}{h}\right) - f'\left(\frac{t + b}{h + a}\right). \end{displaymath} 
Since $f'$ is a strictly increasing function, a quick computation gives
\begin{displaymath} d'_{h}(t) \geq 0 \quad \Longleftrightarrow \quad \frac{t}{h} \leq \frac{b}{a}. \end{displaymath}
So, either $hb/a \notin I_{h}$, and $d_{h}$ is strictly monotone on $I_{h}$, or then $hb/a \in I_{h}$, and the picture looks something like Figure \ref{fig3}. In particular, $d_{h}$ can have zero, one or two zeros on $I_{h}$. Observe that any zero $z$ of $d_{h}$ must satisfy 
\begin{equation}\label{form20} |z - hb/a| \gtrapprox \delta^{\tau}, \end{equation}
since otherwise 
\begin{displaymath} |d'_{h}(z)| = \left|f'\left(\frac{z}{h}\right) - f'\left(\frac{z + b}{h + a}\right)\right| \lesssim \left|\frac{z}{h} - \frac{z + b}{h + a}\right| = \left|\frac{a(z - hb/a)}{h(h + a)}\right| < \delta^{\tau}, \end{displaymath}
contrary to \eqref{transversality}. Also, it is good to keep in mind that if there are two zeros $z_{1},z_{2}$, then $hb/a \in I_{h}$, and $z_{1},z_{2}$ have to be located on different sides of $hb/a$.
\begin{figure}[h!]
\begin{center}
\includegraphics[scale = 0.6]{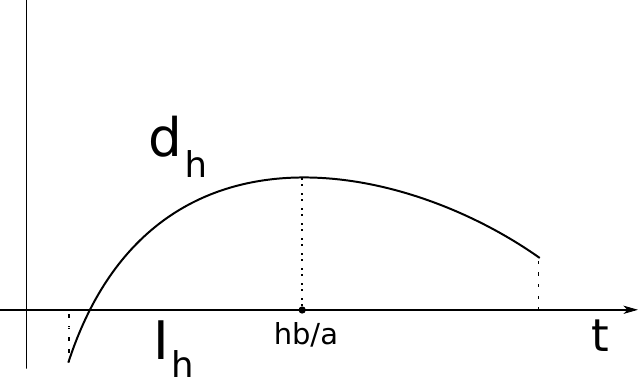}
\caption{The function $d_{h}$.}\label{fig3}
\end{center}
\end{figure} 
Now, define exactly two \emph{special points} $s_{1}(h),s_{2}(h) \in I_{h}$ 
\begin{equation}\label{form21} s_{1}(h) < s_{2}(h) \end{equation}
as follows. 

\begin{itemize}
\item If $d_{h}$ has a zero $z \leq hb/a$, then $s_{1}(h) = z$. Otherwise $s_{1}(h)$ is the left endpoint of $I_{h}$.
\item If $d_{h}$ has a zero $z \geq hb/a$, then $s_{2}(h) = z$. Otherwise $s_{2}(h)$ is the right endpoint of $I_{h}$.
\end{itemize}
Next, the plan is to argue that the set $Z^{h} := \{t \in I_{h} : |d_{h}(t)| \leq \delta^{1/2 + 2\tau}\}$ is contained in two intervals of length $\lesssim \delta^{1/2 + \tau}$, centred at the special points $s_{1}(h)$ and $s_{2}(h)$. First, consider the part
\begin{displaymath} Z^{h}_{\leq hb/a} := \{t \in I_{h} \cap (-\infty,hb/a] : |d_{h}(t)| \leq \delta^{1/2 + 2\tau}\}. \end{displaymath}
The function $d_{h}$ is strictly increasing on $I_{h} \cap (-\infty,hb/a]$, so it can have at most one zero on this interval. If such a zero exists, it is located at $s_{1}(h)$, and $Z^{h}_{\leq hb/a}$ is an interval $I$ around $s_{1}(h)$. Moreover, since $d'_{h}(t) \geq \delta^{\tau}$ for all $t \in I$ by \eqref{transversality}, the length of $I$ is bounded by $\ell(I) \leq 2\delta^{1/2 + \tau}$, as desired. 

If there is no zero of $d_{h}$ on $I_{h} \cap (-\infty,hb/a]$, then the left endpoint of $I_{h}$ is $s_{1}(h)$. Moreover, if $Z^{h}_{\leq hb/a}$ is non-empty, it is an interval $J$ containing $s_{1}(h)$. Once more, $d'_{h}(t) \geq \delta^{\tau}$ for all $t \in J$ by \eqref{transversality}, and this implies that $\ell(J) \leq 2\delta^{1/2 + \tau}$.

A similar argument shows that $\{t \in I_{h} \cap [hb/a,\infty) : |d_{h}| \leq \delta^{1/2 + 2\tau}\}$ is contained in a single interval of length $\lesssim \delta^{1/2 + \tau}$ around one of the special points on $I_{h} \cap [hb/a,\infty)$. Putting the two pieces together, $Z^{h}$ is indeed contained in two intervals of length $\lesssim \delta^{1/2 + \tau}$ centred at the special points. These intervals are denoted by $I_{1}(h)$ and $I_{2}(h)$ (not to be confused with $I_{h_{1}}$ and $I_{h_{2}}$).

Finally, it is time to examine how the special points $s_{1}(h)$ and $s_{2}(h)$ vary as functions of $h$. The desirable conclusion has the form
\begin{equation}\label{IFT} |h_{1} - h_{2}| \leq \delta^{1/2 + 2\tau + C\epsilon} \quad \Longrightarrow \quad |s_{2}(h_{1}) - s_{2}(h_{2})| \lessapprox \delta^{1/2}, \end{equation}
where $C$ is a large enough absolute constant; the same statement holds for $s_{1}(h)$, and the proof is slightly easier. So, assume that $h_{1},h_{2} \in [0,1]$ satisfy \eqref{heightAssumption}, and $|h_{1} - h_{2}| \leq \delta^{1/2 + 2\tau + C\epsilon}$. There are essentially two different cases.

First, it is possible that $s_{2}(h_{1})$ is the right endpoint of $I_{h_{1}}$, and $s_{2}(h_{2})$ is the right endpoint of $I_{h_{2}}$. Then $|s_{2}(h_{1}) - s_{2}(h_{2})| \leq |h_{1} - h_{2}| \leq \delta^{1/2 + 2\tau + C\epsilon}$, which is good.

The second possibility is that at least one of the points $s_{2}(h_{i})$ is a zero of $d_{h_{i}}$. Assume, for instance, that this is the case for $s_{2}(h_{1})$. Then, if $C$ is large enough, one can find a point $t \in I_{h_{2}}$ such that 
\begin{equation}\label{form26} |t - s_{2}(h_{1})| \leq \delta^{1/2 + \tau} \quad \text{and} \quad |d_{h_{2}}(t)| \leq \delta^{1/2 + 2\tau}. \end{equation}
Indeed, one should choose the point $t \in I_{h_{2}}$ closest to $s_{2}(h_{1})$ (so that if $s_{2}(h_{1}) \in I_{h_{2}}$, one would simply pick $t = s_{2}(h_{1})$). Such a point $t$ can be found even at distance $\lesssim \delta^{1/2 + 2\tau + C\epsilon}$ from $s_{2}(h_{1})$, because the endpoints of $I_{h}$ only move at a $1$-Lipschitz rate as $h$ varies. Then, the condition $|d_{h_{2}}(t)| \leq \delta^{1/2 + 2\tau}$ is a consequence of the easy fact that, under the assumption \eqref{heightAssumption}, the mapping $(h,t) \mapsto d_{h}(t)$ is Lipschitz in the following sense:
\begin{displaymath} |d_{h_{2}}(t)| = |d_{h_{1}}(s_{2}(h_{1})) - d_{h_{2}}(t)| \lessapprox \max\{|h_{1} - h_{2}|, |t - s_{2}(h_{1})|\} \lesssim \delta^{1/2 + 2\tau + C\epsilon}. \end{displaymath}
Choosing $C$ counters the implicit constants in the $\lessapprox$ notation and gives \eqref{form26}.

Now, the point $t$ lies in the set $Z^{h_{2}}$, so it is at distance $\lesssim \delta^{1/2 + \tau}$ from either one of the special points $s_{i}(h_{2})$, $i \in \{1,2\}$ by the previous considerations. This implies that
\begin{equation}\label{form22} |s_{2}(h_{1}) - s_{i}(h_{2})| \lesssim \delta^{1/2 + \tau}, \end{equation}
which looks like a little better than required: the surplus $\tau$ will be lost when proving that one can take $i = 2$. 

This is yet another case chase. First, assume that $s_{1}(h_{2})$ is a zero of $d_{h_{2}}$. Then $s_{1}(h_{2}) < h_{2}b/a$. Also, since $s_{2}(h_{1})$ is a zero of $d_{h}$, one has $s_{2}(h_{1}) > h_{1}b/a$, and indeed $s_{2}(h_{1}) - h_{1}b/a \gtrapprox \delta^{\tau}$ by \eqref{form20}. These facts show that
\begin{displaymath} s_{2}(h_{1}) - s_{1}(h_{2}) \geq s_{2}(h_{1}) - h_{2}b/a \gtrapprox \delta^{\tau} - |h_{1}b/a - h_{2}b/a| \geq \delta^{\tau}/2, \end{displaymath}
since $|h_{1} - h_{2}| \leq \delta^{1/2}$ and $b/a \leq s_{2}(h_{1})/h_{1} \leq 1$. In particular, \eqref{form22} is out of the question with $i = 1$, for small enough $\tau > 0$.

So, if $i = 1$, it has to be the case that $s_{1}(h_{2})$ is an endpoint of $I_{h_{2}}$ -- namely the left one. Now, if $s_{1}(h_{2}) < h_{2}b/a$, one can reason exactly as above to show that \eqref{form22} is impossible with $i = 1$. So, one only has to consider the case $s_{1}(h_{2}) \geq h_{2}b/a$. Recall that $s_{i}(h_{2})$ was at distance $\lesssim \delta^{1/2 + \tau}$ from a certain point $t$ with $|d_{h_{2}}(t)| \leq \delta^{1/2 + 2\tau}$ (chosen above \eqref{form22}), which implies that 
\begin{equation}\label{form23} |d_{h_{2}}(s_{i}(h_{2}))| \lessapprox \delta^{1/2 + \tau}. \end{equation}
So, if $i = 1$, the left endpoint $s_{1}(h_{2})$ of $I_{h_{2}}$ has to satisfy \eqref{form23}. Then, by \eqref{transversality}, there exists a zero $z \in I_{h_{2}}$ of $d_{h_{2}}$ with $z - s_{1}(h_{2}) \lessapprox \delta^{1/2}$ \textbf{unless} $I_{h_{2}}$ is too short for this to happen, namely $\ell(I_{h_{2}}) \lessapprox \delta^{1/2}$. In both cases, the special points of $I_{h_{2}}$ are necessarily close to each other, $|s_{1}(h_{2}) - s_{2}(h_{2})| \lessapprox \delta^{1/2}$, which -- combined with \eqref{form22} -- shows that $|s_{2}(h_{1}) - s_{2}(h_{2})| \lessapprox \delta^{1/2}$. Thus, the proof of \eqref{IFT} is complete.

Now, it is time to finish the argument, and cover $\cC \cap (\cC + p) \cap B(0,1)$ by a small number of balls of diameter $\lessapprox \delta^{1/2}$. At the risk of over-repeating an argument, I recall it once more: it has been established that 
\begin{displaymath} \{t \in I_{h} : |d_{h}(t)| \lessapprox \delta\} \subset Z^{h} \end{displaymath}
can be covered by two intervals of length $\lesssim \delta^{1/2}$ centred at $s_{1}(h)$ and $s_{2}(h)$. It follows that $H(h) \cap \cC \cap (\cC + p)$ can be covered by two discs of diameter $\lessapprox \delta^{1/2}$ centred at
\begin{displaymath} \operatorname{centre}_{1}(h) = \left(s_{1}(h),hf\left(\frac{s_{1}(h)}{h}\right),h\right) \text{ and } \operatorname{centre}_{2}(h) := \left(s_{2}(h),hf\left(\frac{s_{2}(h)}{h}\right),h\right). \end{displaymath}
Next, take any interval $H \subset [-1,1]$ such that $|H| \leq \delta^{1/2 + 2\tau + C\epsilon}$, and $h,h + a \geq \delta^{\epsilon}$ for $h \in H$. It follows from \eqref{IFT} (which is the "poor man's implicit function theorem") that 
\begin{displaymath} |\operatorname{centre}_{1}(h_{1}) - \operatorname{centre}_{1}(h_{2})| \lessapprox \delta^{1/2} \end{displaymath}
for all $h_{1},h_{2} \in H$. Consequently, $\cC \cap (\cC + p) \cap \R^{2} \times H$ can be covered by two balls of diameter $\lessapprox \delta^{1/2}$. Finally, one can split the set
\begin{displaymath} B_{+} := B(0,1) \cap \{(t,y,h) : h,h + a \geq \delta^{\epsilon}\} \end{displaymath}
into $\lessapprox \delta^{-1/2 - 2\tau}$, regions of the form $B_{+} \cap \R^{2} \times H$, where $H \subset [-1,1]$ satisfies the requirements above. It follows that $\cC \cap (\cC + p) \cap B_{+}$ can be covered by $\lessapprox \delta^{-1/2 - 2\tau}$ balls of diameter $\lessapprox \delta^{1/2}$, and the proof of Lemma \ref{twoCones} is complete.\end{proof}

\section{Proof of the Three cones lemma}\label{AppendixA}

Recall the statement:

\begin{lemma}[Three cones lemma]\label{threeCones} There is an absolute constant $c \in (0,1)$ such that the following holds for small enough $\delta > 0$. Let $C \subset \R^{3}$ be the cone $C = \{(x,y,z) : x^{2} + y^{2} = z^{2}\}$, and let $p,q \in B(0,1)$ be points satisfying
\begin{displaymath} \min\{|p|,|q|,|p - q|\} \geq \delta^{c}. \end{displaymath}
Write
\begin{displaymath} \cC_{0} := B(C,\delta), \quad \cC_{p} := p + \cC_{0}, \quad \cC_{q} := q + \cC_{0}. \end{displaymath}
Then the intersection
\begin{displaymath} (\cC_{0} \cap \cC_{p} \cap \cC_{q}) \cap B(0,1) \end{displaymath}
is contained in the $\delta^{c}$-neighbourhood of at most two of the lines on $C$. \end{lemma}

It seems likely that the lemma should hold with one line in place of two, but this way the proof is easier. The argument divides into several propositions. I will not write a heuristic overview of them here, because this would essentially be repeating the paragraph "Proof of Lemma \ref{threeCones}" at the end of the paper; in fact, I suggest the reader take a look there before starting with the technicalities.

In order to avoid writing '$B(0,1)$' all the time, the agreement is made that the \textbf{all} the sets below will be intersected with $B(0,1)$. Thus, any claim concerning, say, $\cC_{0} \cap \cC_{p}$ should be interpreted as a claim concerning $\cC_{0} \cap \cC_{p} \cap B(0,1)$ instead. A similar remark concerns the words \emph{taking $c,\delta > 0$ small enough}: these should be inserted anywhere in the text, where they appear needed but missing.

\begin{proposition}\label{task1} Suppose that either $p$ or $q$, say $p$, lies in the $\delta^{1/4}$-neighbourhood of $C$. Then $\cC_{0} \cap \cC_{p}$ (and in particular $\cC_{0} \cap \cC_{p} \cap \cC_{q}$) is contained in the $\delta^{c}$-neighbourhood of a single line on $C$.
\end{proposition}
\begin{proof} Assume, without loss of generality, that $p$ lies in the $\delta^{1/4}$-neighbourhood of the line $\spa(0,1,1) \subset C$. Then $p = (0,r,r) + e$, where $|e| \leq \delta^{1/4}$ and $|r| \gtrsim \delta^{c}$. The idea is to study separately all the intersections $\cC_{0} \cap \cC_{p} \cap H_{t}$, $t \in \R$, where $H_{t}$ is the horizontal plane $H_{t} = \{(x,y,t) : (x,y) \in \R^{2}\} \subset \R^{3}$. Fix $t \in \R$ and make the temporary identification $H_{t} \cong \R^{2}$ (that is, drop off the third component from all vectors on $H_{t}$). Then $\cC_{0} \cap H_{t}$ and $\cC_{p} \cap H_{t}$ are contained in the $\delta$-neighbourhoods of the circles
\begin{displaymath} S_{0} = S((0,0),|t|) \subset \R^{2}  \quad \text{and} \quad S_{p} = S((p_{1},p_{2}),|p_{3} - t|) \subset \R^{2}, \end{displaymath} 
respectively. Since
\begin{displaymath} S((p_{1},p_{2}),|p_{3} - t|) = S((e_{1},r + e_{2}),|r + e_{3} - t|), \end{displaymath}
where $|(e_{1},e_{2},e_{3})| \leq \delta^{1/4}$, one may infer that the $\delta$-neighbourhood $B(S_{p},\delta)$ is contained in the $R\delta^{1/4}$-neighbourhood of the circle $S((0,r),|r - t|)$ for some large enough absolute constant $R \geq 1$. Now, the circles $S(0,|t|)$ and $S((0,r),|r - t|)$ are tangent (either internally or externally) at $(0,t)$, so the intersection of their $R\delta^{1/4}$-neighbourhoods is contained in a small disc $D$ centred at $(0,t)$. The diameter of $D$ depends, of course, on the size of $r$, but choosing $c,\delta > 0$ small enough and assuming $|r| \sim |p| \gtrsim \delta^{c}$ guarantees that $\operatorname{diam}(D) \leq \delta^{c}$. For more details, see the proof of \cite[Lemma 3.1]{Wo}.

Finally, observe that $(0,t,t)$ -- the midpoint of $D$ lifted from $\R^{2}$ to $H_{t}$ -- lies on the line $L = \spa(0,1,1) \subset C$. Repeating the argument above for every $t \in \R$ shows that $\cC_{0} \cap \cC_{p}$ is contained in the $\delta^{c}$-neighbourhood of $L$. \end{proof}

\begin{proposition}\label{metric} Let $A_{1}$ and $A_{2}$ be sets in a metric space $(X,d)$, and let $r,s > 0$. Then
\begin{displaymath} B(A_{1},r) \cap B(A_{2},s) \subset B(B(A_{1},r + s) \cap A_{2},s). \end{displaymath}
\end{proposition}

\begin{proof} Let $x \in B(A_{1},r) \cap B(A_{2},s)$. Choose $a \in A_{1}$, $b \in A_{2}$ such that $d(x,a) \leq r$, $d(x,b) \leq s$. Then $b \in B(A_{1},r + s) \cap A_{2}$, so that $x \in B(B(A_{1},r + s) \cap A_{2},s)$. \end{proof}

\begin{proposition}\label{task2} There is an absolute constant $R \geq 1$ such that the intersections $\cC_{0} \cap \cC_{p}$ and $\cC_{0} \cap \cC_{q}$ are contained in the $R\delta^{1 - c}$-neighbourhoods of the planes
\begin{displaymath} V_{p} := \left\{(x,y,z) : \left((x,y,z) - \frac{(p_{1},p_{2},p_{3})}{2} \right) \cdot (p_{1},p_{2},-p_{3}) = 0\right\} \end{displaymath}
and
\begin{displaymath} V_{q} := \left\{(x,y,z) : \left((x,y,z) - \frac{(q_{1},q_{2},q_{3})}{2} \right) \cdot (q_{1},q_{2},-q_{3}) = 0\right\}. \end{displaymath}
\end{proposition}

\begin{proof} By the previous proposition, it suffices to prove the claim for the intersection $C \cap \cC_{p}$. Note that
\begin{displaymath} \cC_{p} = \bigcup_{r \in B(0,\delta)} p + r + C. \end{displaymath}
We will now prove that  $C \cap (p + r + C)$ is contained in the $R\delta^{1 - c}$-neighbourhood of $V_{p}$ for every $r \in B(0,\delta)$. Using the equation $C = \{(x,y,z) : x^{2} + y^{2} = z^{2}\}$, one can check that $C \cap (p + r + C)$ is contained in the plane
\begin{displaymath} \left\{\left((x,y,z) - \frac{(p_{1},p_{2},p_{3}) + (r_{1},r_{2},r_{3})}{2}\right) \cdot [(p_{1}, p_{2}, -p_{3}) + (r_{1},r_{2},-r_{3})] = 0\right\}. \end{displaymath}
Now, if $(x,y,z) \in B(0,1)$ satisfies the equation above, then it follows from $|r| \leq \delta$ and $p \in B(0,1)$ that
\begin{displaymath} \left| \left((x,y,z)  - \frac{(p_{1},p_{2},p_{3})}{2} \right) \cdot (p_{1},p_{2},-p_{3}) \right| \leq 3\delta. \end{displaymath}
Choose $(x',y',z') \in V_{p}$ such that the difference $(x,y,z) - (x',y',z')$ is parallel to $(p_{1},p_{2},-p_{3})$ (so $(x',y',z')$ is the orthogonal projection of $(x,y,z)$ into $V_{p}$). Then
\begin{align*} |(x,y,z) - (x',y',z')||p| & = |(x,y,z) - (x',y',z')||(p_{1},p_{2},-p_{3})|\\
& = |[(x,y,z) - (x',y',z')] \cdot (p_{1},p_{2},-p_{3})|\\
& = \left| \left((x,y,z)  - \frac{(p_{1},p_{2},p_{3})}{2} \right) \cdot (p_{1},p_{2},-p_{3}) \right| \leq 3\delta, \end{align*} 
proving that $(x,y,z)$ lies in the $(3\delta/|p|)$-neighbourhood of $V_{p}$. Since $|p| \geq \delta^{c}$ by hypothesis, the claim follows.
\end{proof}

For the remainder of the proof, fix $\tau \in (1/2,1)$.

\begin{proposition}\label{task3} Assume that $p,q \notin B(C,\delta^{1/4})$ and $\dist(p,\spa(q)) \leq \delta^{\tau}$. Then the intersection $B(V_{p},R\delta^{1-c}) \cap B(V_{q},R\delta^{1 - c})$ is empty. In particular, the previous lemma implies that
\begin{displaymath} \cC_{0} \cap \cC_{p} \cap \cC_{q} = \emptyset. \end{displaymath}
\end{proposition}

\begin{proof} It suffices to show that the planes $V_{p}$ and $V_{q}$ intersected with $B(0,1)$ are at distance more than $3R\delta^{1-c}$ apart. Let $\xi = q/|q| \in S^{2}$, and write $p = r\xi + e$, where $|e| \leq \delta^{\tau}$, and $|r - |q|| \gtrsim \delta^{1/4}$ (for the latter inequality one uses the assumption $|p - q| \geq \delta^{c}$ with $c \leq 1/4$). Then the equation for the plane $V_{p}$ becomes
\begin{displaymath} \left\{(x,y,z) \cdot (r\xi_{1} + e_{1}, r\xi_{2} + e_{2}, -r\xi_{3} - e_{3}) = \frac{(r\xi_{1} + e_{1})^{2} + (r\xi_{2} + e_{2})^{2} - (r\xi_{3} + e_{3})^{2}}{2}\right\}. \end{displaymath}
This means that if $(x,y,z) \in V_{p}$, then
\begin{align*} (x,y,z) \cdot (\xi_{1},\xi_{2},-\xi_{3}) = r \cdot \frac{\xi_{1}^{2} + \xi_{2}^{2} - \xi_{3}^{2}}{2} \pm O(\delta^{\tau}) = r \cdot \frac{1 - 2\xi_{3}^{2}}{2} \pm O(\delta^{\tau}).    \end{align*} 
On the other hand, if $(x',y',z') \in V_{q}$, then
\begin{displaymath} (x',y',z') \cdot (\xi_{1},\xi_{2},-\xi_{3}) = |q| \cdot \frac{1 - 2\xi_{3}^{2}}{2}. \end{displaymath}
Thus, for $(x,y,z) \in V_{p}$ and $(x',y',z') \in V_{q}$, one finds that
\begin{displaymath} |[(x,y,z) - (x',y',z')] \cdot (\xi_{1},\xi_{2},-\xi_{3})| \geq |r - |q|| \cdot \frac{1 - 2\xi_{3}^{2}}{2} - O(\delta^{\tau}). \end{displaymath}
The assumption $q \notin B(C,\delta^{1/4})$ shows that $\dist(\xi,C) \geq \delta^{1/4}$. Observing that $C \cap S^{2} = \{(t_{1},t_{2},t_{3}) : t_{3} \in \{-1/\sqrt{2},1/\sqrt{2}\}\} \cap S^{2}$, this (and $\xi \in S^{2}$) implies further that 
\begin{displaymath} \dist(\xi_{3},\{-1/\sqrt{2},1/\sqrt{2}\}) \gtrsim \delta^{1/4}. \end{displaymath}
Since the derivative of the mapping $t \mapsto 1 - 2t^{2}$ stays bounded away from zero near $t = \pm 1/\sqrt{2}$, one may infer that $|(1 - 2\xi_{3}^{2})/2| \gtrsim \delta^{1/4}$. All in all, for small enough $\delta > 0$,
\begin{displaymath} |(x,y,z) - (x',y',z')| \geq |[(x,y,z) - (x',y',z')] \cdot (\xi_{1},\xi_{2},-\xi_{3})| \gtrsim \delta^{1/2}. \end{displaymath}
Assuming that $c < 1/2$, the term on the right hand side dominates $3R\delta^{1 - c}$ for small enough $\delta > 0$. This proves that $\dist(V_{p} \cap B(0,1),V_{q} \cap B(0,1)) \geq 3R\delta^{1 - c}$, and so the two $R\delta^{1 - c}$-neighbourhoods cannot intersect inside $B(0,1)$.
\end{proof} 

\begin{proposition}\label{task4} Assume that $\dist(p,\spa(q)) \geq \delta^{\tau}$. Then, for small enough $c,\delta > 0$, the intersection $B(V_{p},R\delta^{1 - c}) \cap B(V_{q},R\delta^{1 - c})$ is contained in the $\delta^{c}$-neighbourhood of the the line $V_{p} \cap V_{q}$.
\end{proposition}

\begin{proof} Translating if necessary, one may assume that the line $L = V_{p} \cap V_{q}$ passes through the origin. Let $y \in B(V_{p},R\delta^{1 - c}) \cap B(V_{q},R\delta^{1 - c})$. Then $y = l + x$, where $l \in L$ and $x \in L^{\perp} = \spa\{\bar{p},\bar{q}\}$. Here $\bar{p} = (p_{1},p_{2},-p_{3})/|p|$ and $\bar{q} = (q_{1},q_{2},-q_{3})/|q|$ are normal to $V_{p}$ and $V_{q}$, respectively. Then, since $\{\bar{p},(\bar{q} - (\bar{p} \cdot \bar{q})\bar{p})/|\bar{q} - (\bar{p} \cdot \bar{q})\bar{p}|\}$ is an orthonormal basis for $\spa\{\bar{p},\bar{q}\}$, one sees that
\begin{align*} |x| \sim |x \cdot \bar{p}| + \left|\frac{x \cdot (\bar{q} - (\bar{p} \cdot \bar{q})\bar{p})}{|\bar{q} - (\bar{p} \cdot \bar{q})\bar{p}|} \right| \leq |x \cdot \bar{p}| + \frac{|x \cdot \bar{q}|}{|\bar{q} - (\bar{p} \cdot \bar{q})\bar{p}|} + \frac{|x \cdot \bar{p}|}{|\bar{q} - (\bar{p} \cdot \bar{q})\bar{p}|}. \end{align*}
Here $|x \cdot \bar{p}|,|x \cdot \bar{q}| \leq R\delta^{1 - c}$, since, for instance, $|x \cdot \bar{p}| = \dist(y,V_{p}) \leq R\delta^{1 - c}$. On the other hand $|\bar{q} - (\bar{p} \cdot \bar{q})\bar{p}| \geq \delta^{\tau}$ by assumption, so one obtains $\dist(y,L) = |x| \lesssim \delta^{1 - c - \tau}$. Hence, the claim is true as long as $c < 1 - c - \tau$.

 \end{proof}

\begin{proposition}\label{task5} Let $L$ be an arbitrary line in $\R^{3}$. Then the intersection $B(L,\delta^{c}) \cap \cC_{0}$ is contained in the $\delta^{c^{2}/5}$-neighbourhood of at most two lines on $C$. 
\end{proposition}


\begin{proof}Let $L$ be the line $L = \{r\xi + p : r \in \R\}$, where $\xi \in S^{2}$ and $p \in \R^{3}$. Assume first that $\xi$ forms a small angle with one of the lines on $C$, say $\dist(\xi,C) \leq \delta^{c/4}$. Then, if $q \in B(L,\delta^{c})$, one may conclude that $B(L,\delta^{c}) \subset q + B(C,\delta^{c/5})$ for small enough $\delta > 0$. Thus, assuming that $B(L,\delta^{c})$ intersects $\cC_{0}$ at even one point, say $q \in \cC_{0}$, then certainly $B(L,\delta^{c}) \cap \cC_{0} \subset (q + B(C,\delta^{c/5})) \cap B(C,\delta^{c/5})$. But now Proposition \ref{task1} is applicable and shows that $B(C,\delta^{c/5}) \cap (q + B(C,\delta^{c/5}))$ is contained in the $\delta^{c^{2}/5}$-neighbourhood of a single line on $C$.

Next, assume that $\dist(\xi,C) \geq \delta^{c/4}$. By Proposition \ref{metric}, it suffices to prove that $B(L,\delta^{c}) \cap C$ is contained in the union of two small balls centred at points on $C$. The neighbourhood $B(L,\delta^{c})$ is the union of the lines $L_{q} := \{r\xi + q : r \in \R\}$, where $q \in p + B(0,\delta^{c})$. We may explicitly find the (at most) two points on $L_{q} \cap C$, since such points must satisfy
\begin{displaymath} (r\xi_{1} + q_{1})^{2} + (r\xi_{2} + q_{2})^{2} - (r\xi_{3} + q_{3})^{2} = 0, \end{displaymath}
amounting to
\begin{displaymath} r = \frac{-2(\xi_{1}q_{1} + \xi_{2}q_{2} - \xi_{3}q_{3})\pm\sqrt{4(\xi_{1}q_{1} + \xi_{2}q_{2} - \xi_{3}q_{3})^{2} - 4(\xi_{1}^{2} + \xi_{2}^{2} - \xi_{3}^{2})(q_{1}^{2} + q_{2}^{2} - q_{3}^{2})}}{2(\xi_{1}^{2} + \xi_{2}^{2} - \xi_{3}^{2})}. \end{displaymath}
The denominator is $\gtrsim \delta^{c/4}$, by the assumption $\dist(\xi,C) \geq \delta^{c/4}$. The numerator, on the other hand is $1/2$-Hölder continuous with respect to moving the point $q = (q_{1},q_{2},q_{3})$ around. So, when $q$ ranges in $p + B(0,\delta^{c})$, the solutions $r = r(q)$ can vary only inside intervals of length $\lesssim \delta^{-c/4} \cdot \delta^{c/2} = \delta^{c/4}$. This implies that the intersection $B(L,\delta^{c}) \cap C$ is contained in two balls of radius $\lesssim \delta^{c/4}$. \end{proof} 

\begin{proof}[Proof of Lemma \ref{threeCones}] The lemma follows by combining the propositions. If either $p$ or $q$ lies very close to the surface $C$, one is instantly done by Proposition \ref{task1}. If both points lie far from $C$, then Proposition \ref{task3} implies that either $\cC_{0} \cap \cC_{p} \cap \cC_{q}$ is empty, or $p$ does not lie close to the line spanned by $q$. In the latter case, the intersection $\cC_{0} \cap \cC_{p} \cap \cC_{q}$ is contained in the small neighbourhood of a single line in $\R^{3}$, according to Proposition \ref{task4}. Finally, by Proposition \ref{task5}, the intersection of any such neighbourhood with $\cC_{0}$ is contained in the neighbourhood of at most two lines on $C$, as claimed. \end{proof}

\end{document}